\documentclass{amsart}
\usepackage[latin1]{inputenc}
\usepackage{amssymb}
\usepackage{latexsym}
\usepackage{amscd}
\usepackage{longtable}
\usepackage{amsmath}
\usepackage{array}
\usepackage[all]{xy}
\usepackage{a4}

\usepackage[OT2,T1]{fontenc}
\DeclareSymbolFont{cyrletters}{OT2}{wncyr}{m}{n}
\DeclareMathSymbol{\Sha}{\mathalpha}{cyrletters}{"58}

\newtheorem{defn0}{Definition}[section]
\newtheorem{prop0}[defn0]{Proposition}
\newtheorem{thm0}[defn0]{Theorem}
\newtheorem{lemma0}[defn0]{Lemma}
\newtheorem{claim0}[defn0]{Claim}
\newtheorem{corollary0}[defn0]{Corollary}
\newtheorem{example0}[defn0]{Example}
\newtheorem{remark0}[defn0]{Remark}
\newtheorem{assumption0}[defn0]{Assumption}
\newtheorem{conjecture0}[defn0]{Conjecture}
\newtheorem{notation0}[defn0]{Notation}
\newtheorem{question0}[defn0]{Question}

\newenvironment{definition}{\begin{defn0}\rm}{\end{defn0}}
\newenvironment{proposition}{\begin{prop0}}{\end{prop0}}
\newenvironment{theorem}{\begin{thm0}}{\end{thm0}}
\newenvironment{lemma}{\begin{lemma0}}{\end{lemma0}}

\newenvironment{corollary}{\begin{corollary0}}{\end{corollary0}}

\newenvironment{remark}{\begin{remark0}\rm}{\end{remark0}}

\newcommand{\Gal}{{\mathrm {Gal}}}

\newcommand{\Trace}{{\mathrm{Tr}}}
\newcommand{\M}{\mathrm{M}}
\newcommand{\Aut}{\mathrm{Aut}}

\newcommand{\Br}{\mathrm{Br}}

\newcommand{\GL}{{\mathrm{GL}}}
\newcommand{\End}{{\mathrm{End}}}

\newcommand{\Z}{{\mathbb Z}}
\newcommand{\A}{{\mathbb A}}
\newcommand{\Q}{{\mathbb Q}}
\newcommand{\C}{{\mathbb C}}
\newcommand{\R}{{\mathbb R}}

\newcommand{\N}{{\mathbb N}}

\newcommand{\cO}{{\mathcal O}}

\newcommand{\Res}{{\mathrm {Res}}}

\newcommand{\ra}{{\rightarrow}}

\newcommand{\Hom}{{\mathrm {Hom}}}

\title{Galois action on $\bar\Q$-isogeny classes \\ of abelian $L$-surfaces with quaternionic multiplication}
\author{Santiago Molina
}

\begin{document}
%%%%%%%%%%%%%%%%%%%%%%%%%%%%%%%%%%%%%%%%%%%%%

\address{Department of Mathematics, Universit\"at Bielefeld, Bielefeld (Germany)}
\email{santimolin@gmail.com}

\maketitle
%%%%%%%%%%%%%%%%%%%%%%%%%%%%%%%%%%%%%%%%%%%%%
%\vskip 1cm

\begin{abstract} 
We construct a projective Galois representation attached to an abelian $L$-surface with quaternionic multiplication, describing the Galois action on its Tate module.
We prove that such representation characterizes the Galois action on the isogeny class of the abelian $L$-surface, seen as a set of points of certain Shimura curves.
\end{abstract}

\section{Introduction}

Let $L$ be a number field.
An abelian variety
$A/\bar \Q$ is called an \emph{abelian $L$-variety} if for each
$\sigma\in\Gal(\bar \Q/L)$ there exists an isogeny
$\mu_\sigma:A^\sigma\ra A$ with the following compatibility condition:
$\psi\circ\mu_\sigma= \mu_\sigma\circ\psi^\sigma$, for all
$\psi\in\End(A)$. 
In this note, we deal with the two dimensional situation and the so-called \emph{fake elliptic curves} or \emph{abelian surfaces with quaternionic multiplication (QM)}, namely, pairs $(A,\imath)$ where $A$ is an abelian surface and $\imath$ is an embedding of a quaternion order $\cO$ into its endomorphism ring. We remark that, by setting $A=E\times E$, where $E$ is an elliptic curve, and $\imath$ the obvious embedding of $\cO=\M_2(\Z)$ into $\End(E\times E)$, the theory of elliptic curves can be seen as a special case.
In this scenario, we say that $(A,\imath)$ is an \emph{abelian $L$-surface with QM} if, for any $\sigma\in\Gal(\bar \Q/L)$, there exists an isogeny
$\mu_\sigma:A^\sigma\ra A$ as well, 
%such that $\imath(\alpha)\circ\mu_\sigma= \mu_\sigma\circ\imath(\alpha)^\sigma$, for all $\alpha\in\cO$.
but with the above compatibility condition only satisfied for the image of the quaternionic multiplication $\imath$.

The moduli problem that classifies abelian surfaces with QM is solved by the classical \emph{Shimura curves} $X_\Gamma$. If $P\in X_\Gamma$ corresponds to an abelian $L$-surface with QM $(A,\imath)$, we can interpret the $\bar\Q$-isogeny class of $(A,\imath)$ as a set of points $[P]\subset X_\Gamma(\bar\Q)$. Directly from the definition, we can deduce that this set of points $[P]$ is stable under the action of $\Gal(\bar\Q/L)$. The main aim of this paper is to construct a projective Galois representation $\rho$ attached to $(A,\imath)$ of the form
\[
\rho:\Gal(\bar \Q/L)\longrightarrow (\cO\otimes\A_f)^\times/(\End(A,\imath)\otimes\Q)^\times,
\]
where $\A_f$ is the ring of finite adeles and $\End(A,\imath)$ is the set of endomorphisms commuting with every element of the image of $\imath$, and to give a concise description of the Galois action on the set of points $[P]$ through this representation $\rho$.
With this in mind, we have to give first a description of the $\bar\Q$-isogeny class $[P]$. In \S \ref{sec3} we show that there is a bijective correspondence between the set $[P]\subset X_\Gamma$ and the double coset space 
\[\Gamma\backslash(\cO\otimes\A_f)^\times/(\End(A,\imath)\otimes\Q)^\times,\]
for some compact open subgroup $\Gamma$. This provides a purely algebraic description of the $\bar\Q$-isogeny class of $(A,\imath)$.
Finally, in \S \ref{sec4}, we introduce the main result of this note (Theorem \ref{TeoGalact}): The Galois action of $\Gal(\bar\Q/L)$ on the $\bar\Q$-isogeny class $[P]\subset X_\Gamma$ is given by the following map
\[
\Gal(\bar\Q/L)\times\Gamma\backslash(\cO\otimes\A_f)^\times/(\End(A,\imath)\otimes\Q)^\times\longrightarrow\Gamma\backslash(\cO\otimes\A_f)^\times/(\End(A,\imath)\otimes\Q)^\times\]\[
\qquad\qquad\qquad(\sigma,[b])\qquad\longmapsto\qquad[b\rho(\sigma)],
\]
where $[b]$ denotes the class of $b\in(\cO\otimes\A_f)^\times$ in the double coset space $\Gamma\backslash(\cO\otimes\A_f)^\times/(\End(A,\imath)\otimes\Q)^\times$.
%In particular, this result describes the Galois action on $\bar\Q$-isogeny classes of abelian $\Q$-surfaces with QM in terms of the projectivization of the product of the $\ell$-adic representations attached to the corresponding abelian varieties of $\GL_2$-type.

The current interest on abelian $L$-varieties began, with $L=\Q$, when K. Ribet observed that non-CM absolutely simple factors of the modular Jacobians $J_1(N)$ are in fact abelian $\Q$-varieties \cite{Rib}. Actually, after the proof of Serre's conjecture on representations of $\Gal(\bar\Q/\Q)$ \cite[3.2.4?]{Ser}, every so-called \emph{building block} (namely a $\Q$-variety whose endomorphism algebra is a central division algebra over a totally real number field $F$ with Schur index $t=1$ or $t=2$ and $t[F:\Q]=\dim A$), is an absolutely simple factor up to isogeny of a modular Jacobian $J_1(N)$.

In order to explore the relation of the representation $\rho$ with modularity, one realizes that the construction of $\rho$ given in \S \ref{sec2} imitates the classical construction of the $\ell$-adic Galois representation of an elliptic curve defined over $L$ or the projective $\ell$-adic Galois representation attached to an elliptic $\Q$-curve. 
In fact, in the trivial case $A=E\times E$ and $E$ one of such objects, $\rho$ is the projectivization of the product over all $\ell$ of the corresponding classical $\ell$-adic representations. Something analogous happens when $(A,\imath)$ is a building block, namely, a non-CM abelian $\Q$-surface with QM by an order in a division algebra. We know that $A$ is an absolutely simple factor of an abelian variety $A_{\GL_2}$ defined over $\Q$ of $\GL_2$-type.
In \S \ref{sec8} we show that $\rho$ is the projectivization of the product over all $\ell$ of the $\ell$-adic Galois representations attached to the abelian variety $A_{\GL_2}/\Q$.

Since in case $L=\Q$ the projective representation $\rho$ is related to well known classical $\ell$-adic representations, we expect the norm of $\rho$ to be characterized by the cyclotomic character. In \S \ref{sec6} we introduce the dual of an abelian surface with QM and we describe the Weil pairing attached to it. The given explicit description of the Weil pairing allows us to compute the norm of $\rho$ and to prove that it is indeed provided by the cyclotomic character (Theorem \ref{normrho}).

Apart from the interesting relations between $\rho$ and certain abelian varieties of $\GL_2$-type in the non-CM case, the CM case is very interesting as well. One can prove that any CM abelian surface with QM (that we know is modular) is in fact an abelian $\Q$-surface with QM (Proposition \ref{PropCM}).
Moreover, the corresponding points in $X_\Gamma$ classifying the $\bar\Q$-isogeny class of a CM abelian ($\Q$-)surface with QM are classical \emph{Heegner points}. This implies, by \emph{Shimura's reciprocity law}, that the Galois action of $\Gal(\bar\Q/\Q)$ on the set $[P]\subset X_\Gamma(\bar\Q)$ of Heegner points is described via Class Field Theory. Using this fact,
we prove in Proposition \ref{PropCM} (as a direct consequence of Theorem \ref{TeoGalact}) that, in the CM case, the projective representation $\rho$ factors through the inverse of the Artin map. This provides a complete description of $\rho$ in this case.

\subsection{Notation}

Let $\hat\Z$ denote the completion of $\Z$, hence $\hat\Z=\varprojlim(\Z/N\Z)$.
Let $\A_f$ denote the ring of finite adeles, namely $\A_f=\hat\Z\otimes\Q$.
Note that $\Q/\Z=\varinjlim \Z/N\Z$, therefore 
\begin{eqnarray*}
\End(\Q/\Z)&=&\Hom(\varinjlim (\Z/N\Z),\Q/\Z)=\varprojlim\Hom(\Z/N\Z,\Q/\Z)\\
&=&\varprojlim\Hom(\Z/N\Z,\Z/N\Z)=\varprojlim (\Z/N\Z)=\hat\Z.
\end{eqnarray*}
Write $\A_\Q$ for the ring of adeles of $\Q$.

Let $B$ be an indefinite quaternion algebra over $\Q$ of discriminant $D$, and let $\cO$ be an Eichler order in $B$.
We denote by $\Trace$ and ${\rm Norm}$ the reduced trace and the reduced norm in $B$ respectively. Namely, for any $b\in B$, we have $\Trace(b)=b+\bar b$ and ${\rm Norm}(b)=b\bar b$, where $\bar b$ stands for the conjugate of the quaternion $b$.
 
Let $G$ be the group scheme over $\Z$ such that $G(R)=(\cO^{opp}\otimes R)^\times$ for all rings $R$, where $\cO^{opp}$ is the opposite algebra to $\cO$. Note that the group $G(\A_f)$ does not depend on the Eichler order $\cO$ chosen since it is maximal locally for all but finitely many places.

Write $\hat\cO=\cO\otimes\hat\Z$, then me have the isomorphism $\hat\cO=\varprojlim(\cO/N\cO)$. Moreover, $\varinjlim(\cO/N\cO)=B/\cO$ as left $\cO$-modules. Applying the above argument, we have $\End_\cO(B/\cO)=\hat\cO^{opp}$, where $\hat\cO^{opp}$ acts on $B/\cO$ on the right. Hence we can identify $G(\A_f)=(\End_\cO(B/\cO)\otimes\Q)^\times$.

Given an abelian variety $A$ defined over $\C$, we will denote by $\End(A)$ the algebra of endomorphisms of $A$ defined over $\C$. Note that if $A$ admits a model over $\bar\Q$, then $\End(A)=\End_{\bar\Q}(A)$.
Throughout this paper, we will denote  $\End^0:=\End\otimes\Q$.

%Given an isogeny between abelian varieties $\mu:A\rightarrow A'$, we will denote by $\deg(\mu)$ its degree and by $\mu^\vee$ its dual isogeny.

\section{Abelian $L$-surfaces with Quaternionic Multiplication}\label{sec2}

Let $F$ be a field. An \emph{abelian surface with QM by $\cO$} over $F$ is a pair $(A,\imath)$ where $A/F$ is an abelian surface and $\imath$ is an embedding $\cO\hookrightarrow \End_F(A)$, optimal in sense that $\imath(B)\cap\End_F(A)=\imath(\cO)$, and such that every endomorphism $\imath(\alpha)$, $\alpha\in\cO$, is defined over $F$. If the order $\cO$ is clear by the context, we will call them just \emph{QM-abelian surfaces}.
Let us consider the subring
\[
\End(A,\imath)=\{\lambda\in\End_F(A):\;\lambda\circ\imath(o)=\imath(o)\circ\lambda,\;\mbox{for all}\;o\in\cO\}.
\]
%Let us denote by $\End^0(A,\imath):=\End(A,\imath)\otimes\Q$.
If $(A,\imath)$ is defined over $\C$, $\End^0(A,\imath)$ can be either $\Q$ or an imaginary quadratic field $K$, in this last situation we say that $(A,\imath)$ has \emph{complex multiplication (CM)}. 

\begin{definition}
Two abelian surfaces $(A,\imath)$ and $(A',\imath')$ with QM by $\cO$ defined over $F$ are \emph{isogenous} or \emph{$\cO$-isogenous} if there exist an isogeny $\mu:A'\rightarrow A$ defined over $F$ satisfying $\mu\circ\imath'(\alpha)=\imath(\alpha)\circ\mu$, for all $\alpha\in\cO$. We will call such an isogeny $\mu:(A',\imath')\rightarrow(A,\imath)$ a \emph{$\cO$-isogeny}. 
\end{definition}

Assume that $(A,\imath)$ is defined over $\C$ and let $\hat T(A)=\Hom(A_{tor},\Q/\Z)$ be its Tate module. Since $\hat T(A)$ is a $\A_f$-module of rank 4 with $\End(\hat T(A))=\End(A)\otimes\A_f$ and $\imath(\cO)=\imath(B)\cap\End(A)$, we conclude that $\hat T(A)\simeq \cO\otimes\A_f$ and $A_{tor}\simeq B/\cO$ as $\cO$-modules. This implies that, for any $\cO$-isogeny $\mu:(A,\imath)\rightarrow (A',\imath')$, we have an isomorphism $\ker(\mu)\simeq I_\mu/\cO$ as $\cO$-modules, for some left fractional $\cO$-ideal $I_\mu$. We define $\deg(\mu)$ to be ${\rm Norm}(I_\mu)^{-1}$. With this definition, the multiplication-by-$n$ $\cO$-isogeny has degree $\deg(n)=n^2$, instead of $n^4$. Moreover, we have an inclusion $\ker(\mu)\subseteq\ker(\deg(\mu))$ provided by 
\[
\ker(\mu)=I_\mu/\cO\subseteq{\rm Norm}(I_\mu)\cO/\cO=\ker(\deg(\mu)).
\]
This implies that there exists a $\cO$-isogeny $\mu^\vee:(A',\imath')\rightarrow(A,\imath)$ such that $\mu\circ\mu^\vee=\mu^\vee\circ\mu=\deg(\mu)$ and $\deg(\mu)=\deg(\mu^\vee)$. We call such $\cO$-isogeny the \emph{dual $\cO$-isogeny} of $\mu$. 

\begin{definition}
Let $L/\Q$ be a number field.
An \emph{abelian $L$-surface with QM by $\cO$} is an abelian surface with QM by $\cO$ $(A,\imath)$ over $\bar\Q$ such that, for all $\sigma\in\Gal(\bar\Q/L)$, there exist an isogeny $\mu_\sigma:A^\sigma\rightarrow A$ (defined over $\bar\Q$), such that $\mu_\sigma\circ\imath(o)^\sigma=\imath(o)\circ\mu_\sigma$ for all $o\in\cO$. Equivalently, $(A^\sigma, \imath^\sigma)$ and $(A,\imath)$ are $\cO$-isogenous for all $\sigma\in\Gal(\bar\Q/L)$, where $\imath^\sigma$ is defined by $\imath^\sigma(o):=\imath(o)^\sigma$ for $o\in\cO$.
%We say that the abelian $L$-surface with QM by $\cO$ is  \emph{strict} if we can choose the isogenies $\mu_\sigma$ satisfying $\mu_\sigma\circ\lambda^\sigma=\lambda\circ\mu_\sigma$ for all $\lambda\in\End(A)$.
\end{definition}

Given an abelian $L$-surface $(A,\imath)$ with QM we shall construct a map 
\[
\rho_{(A,\imath,\varphi)}: \Gal(\bar L/L) \longrightarrow G(\A_f)/\End^0(A,\imath)^\times
\]
that describes the Galois action on the Tate module.
In order to do this, we will fix an isomorphism $\varphi:A_{tor}\rightarrow B/\cO$. 
The following result shows that to choose such an isomorphism $\varphi$ is equivalent to choose a basis $\{\varphi_1,\varphi_2,\varphi_3,\varphi_4\}$ of the Tate module $\hat T(A)$.
\begin{lemma}\label{equiTatemod}
Given a basis $\{\varphi_i\}_{i=1\cdots 4}$ of $\hat T(A)$, there exists a $\hat\Z$-basis $\{e_i\}_{i=1\cdots 4}$ of $\hat\cO$ such that 
\[
\varphi:A_{tor}\longrightarrow B/\cO\simeq\hat B/\hat\cO;\qquad P\longmapsto \sum_{i=1\cdots 4}\varphi_i(P)e_i
\]
is a $\cO$-module isomorphism. Analogously, given a $\cO$-module isomorphism $\varphi:A_{tor}\rightarrow B/\cO$, any $\hat\Z$-basis $\{e_i\}_{i=1\cdots 4}$ of $\hat\cO$ provides a basis $\{\varphi_i\}_{i=1\cdots 4}$ of $\hat T(A)$ satisfying $\varphi(P)=\sum_{i=1}^4\varphi_i(P)e_i$.
\end{lemma}
\begin{proof}
Since $\{\varphi_i\}_{i=1\cdots 4}$ is a basis, the map $A_{tor}\rightarrow(\Q/\Z)^4;\; P\mapsto (\varphi_i( P))_i$ is an isomorphism. Then there exists a unique sequence $\{P_n\}_{n\in\N}\subset A_{tor}$ such that $\varphi_i(P_n)=\frac{1}{n}\delta_{1}^i$, where $\delta_i^j$ is the Kronecker delta. 
Since $A$ has QM by $\cO$, we know that $A[n]\simeq\cO/n\cO$. Thus there exists $e_j\in\hat\cO=\underleftarrow{\lim}\cO/n\cO$ such that $\varphi_i(\imath(e_j)P_n)=\frac{1}{n}\delta_{j}^i$ (in particular $e_1=1$). We claim that $\{e_i\}_{i=1\cdots 4}$ form a $\hat\Z$-basis for $\hat\cO$. Indeed, %$\sum_{i=1\cdots 4}n_ie_i=0$ for some $n_i\in\Z$ implies that, for all $n\in\N$,
%\[
%0=\sum_{i=1\cdots 4}n_i\varphi_j(\iota(e_i)P_n)=\sum_{i=1\cdots 4}\frac{n_i}{n}\delta_i^j=\frac{n_j}{n}.
%\]
%Therefore $n\mid n_j$ for all $n\in\N$, hence $n_j=0$. On the other hand, 
for any $\alpha\in\hat\cO$, there exists $\hat n_i\in\hat\Z$ such that $\varphi_i(\imath(\alpha)P_n)=\frac{n_i}{n}$, where $n_i=\hat n_i \;{\rm mod}\; n$, thus
\[
\varphi_i(\imath(\alpha)P_n)-\varphi_i\left(\imath\left(\sum_{j=1\cdots 4}\hat n_je_j\right)P_n\right)=\frac{n_i}{n}-\sum_{j=1\cdots 4}\frac{n_j}{n}\delta_i^j=0.
\]
Since $\{\varphi_i\}_{i=1\cdots 4}$ form a basis and $P_n$ generates $A[n]$ as $\cO$-module, we conclude $\alpha=\sum_{j=1\cdots 4}\hat n_je_j$ and $\{e_i\}_{i-1\cdots 4}$ form a $\hat\Z$-basis for $\hat\cO$.

Finally, we consider the well defined morphism $\varphi$ and let $\alpha=\sum_{j=1\cdots 4}\hat n_je_j\in\hat\cO$, then
\[
\varphi(\imath(\alpha)P_n)=\sum_{i,j=1\cdots 4}n_j\varphi_i\left(\imath(e_j)P_n\right)e_i=\sum_{i,j=1\cdots 4}n_j\frac{1}{n}\delta_i^je_i=\frac{1}{n}\sum_{i=1\cdots 4}n_ie_i=\alpha\varphi(P_n).
\]
Since $P_n$ generates $A[n]$ as a $\cO$-module, we conclude that $\varphi$ is a $\cO$-module isomorphism.

Analogously, given a $\cO$-module isomorphism $\varphi:A_{tor}\rightarrow B/\cO$ and given a $\hat\Z$-basis $\{e_i\}_{i=1\cdots 4}$ of $\hat\cO$, we define the morphism $\varphi_i(P)=x_i$, where $\varphi(P)=\sum_{i=1\cdots 4}x_ie_i$. 
It is clear that $\{\varphi_i\}_{i=1\cdots 4}$ provides a $\hat\Z$-basis of $\Hom(A_{tor},\Q/\Z)$.
\end{proof}

%Let us denote by 
%\[
%V(P):=(\varphi_1(P),\varphi_2(P),\varphi_3(P),\varphi_4(P))\in(\Q/\Z)^4,
%\]
%A CHOICE OF THE BASIS MAY PROVIDE AN ISOMORPHISM $\hat T(A)\simeq\Hom_\cO(A_{tor},B/\cO)$???
%From now on, we fix basis $\{\varphi_i\}_{i=1\cdots 4}$ of $\hat T(A)$. By the above lemma, this fixes a $\cO$-module isomorphism $\varphi:A_{tor}\rightarrow B/\cO$.

Since $(A,\imath)$ is defined over $\bar\Q$, we can fix a number field $M$ and a model of $A$ over $M$ such that any endomorphism is defined over $M$. We denote such a model also by $(A,\imath)$ by abuse of notation. We fix a set of $\cO$-isogenies $\mu=\{\mu_\sigma:(A^\sigma,\imath^\sigma)\rightarrow (A,\imath),\;\sigma\in\Gal(M/L)\}$ and assume, after extending $M$ if necessary, that every $\cO$-isogeny in $\mu$ is also defined over $M$. 

Let us consider the endomorphism on $B/\cO$ given by 
\[
\varphi( P)\longmapsto\varphi(\mu_\sigma(P^\sigma)).
\]
Such endomorphism commute with the action of $\cO$, indeed, for any $\alpha\in\cO$,
\[
\varphi(\mu_\sigma((\imath(\alpha)P)^\sigma))=\varphi(\mu_\sigma(\imath(\alpha)^\sigma(P^\sigma)))=\varphi(\imath(\alpha)\mu_\sigma(P^\sigma))
=\alpha\varphi(\mu_\sigma(P^\sigma)).
\]
Hence it corresponds to an element of $\End_\cO^0(B/\cO)^\times$ since $\mu_\sigma$ has finite kernel. Once we identify $\End_\cO^0(B/\cO)^\times$ with $G(\A_f)$ acting on $B/\cO$ on the right (provided that $\cO$ acts on $B/\cO$ on the left), we deduce that there exists $\rho^\mu_{(A,\imath,\varphi)}(\sigma)\in G(\A_f)$ such that 
\[
\varphi(\mu_\sigma(P^\sigma))=\varphi(P)\rho^\mu_{(A,\imath,\varphi)}(\sigma),\;\;\mbox{for all }P\in A_{tor}.
\]
We have obtained a map 
%Moreover, since $\mu_\sigma$ and conjugation by $\sigma\in\Gal(\bar L/L)$ are isomorphisms between $A^\sigma[\ell^\infty]:=\cup_nA^\sigma[\ell^n]$ and $A[\ell^\infty]:=\cup_nA[\ell^n]$ for almost all $\ell$, we have 
\[
\rho^\mu_{(A,\imath,\varphi)}: \Gal(\bar \Q/L) \longrightarrow G(\A_f).
\]
This map may depend on the choice of the set of $\cO$-isogenies $\mu$, nevertheless we can consider the quotient $G(\A_f)/\End^0(A,\imath)^\times$, where $\End^0(A,\imath)^\times$ is embedded in $G(\A_f)$ by means of the natural embedding
\[
\varphi^\ast:\End^0(A,\imath)^\times\hookrightarrow G(\A_f)=\End^0_\cO(B/\cO)^\times;\qquad \varphi^\ast(\lambda)=\lambda^\ast=\varphi\circ\lambda\circ\varphi^{-1}.
\]
%(note that by definition $\varphi^\ast\End(A,\imath)\subset\End_\cO(B/\cO)$), 
Hence the composition with the quotient map, gives rise to a map of the form:
\[
\rho_{(A,\imath,\varphi)}: \Gal(\bar \Q/L) \longrightarrow G(\A_f)/\End^0(A,\imath)^\times.
\]

\begin{lemma}
The map $\rho_{(A,\imath,\varphi)}$ is independent on the choice of the set of $\cO$-isogenies $\mu$ and the choice of the model in the $\bar\Q$-isomorphism class of $(A,\imath)$. 
\end{lemma}
\begin{proof}
Let $\mu'=\{\mu_\sigma':(A^\sigma,\imath^\sigma)\rightarrow(A,\imath),\;\sigma\in\Gal(M/L)\}$ be another set of $\cO$-isogenies (it can be defined over another Galois extension $M'$ but we can extend both sets trivially). Then $\lambda_\sigma:=\frac{1}{\deg(\mu_\sigma)}\mu_\sigma'\circ\mu_\sigma^\vee\in\End^0(A,\imath)^\times$. Hence we have,
\begin{eqnarray*}
\varphi(P)\rho^{\mu'}_{(A,\imath,\varphi)}(\sigma)&=&\varphi(\mu'_\sigma(P^\sigma))=\varphi(\lambda_\sigma(\mu_\sigma(P^\sigma)))=\varphi(\mu_\sigma(P^\sigma))\lambda_\sigma^\ast\\
&=&\varphi(P)\rho^\mu_{(A,\imath,\varphi)}(\sigma)\lambda_\sigma^\ast.
\end{eqnarray*}
Thus $\rho^{\mu'}_{(A,\imath,\varphi)}(\sigma)=\rho^{\mu}_{(A,\imath,\varphi)}(\sigma)\lambda_\sigma^\ast$ and $$\rho^{\mu'}_{(A,\imath,\varphi)}(\sigma)\End^0(A,\imath)^\times=\rho^\mu_{(A,\imath,\varphi)}(\sigma)\End^0(A,\imath)^\times,$$
which proves our first assertion.

Assume that we have another model $(A',\imath')$ over $M'$. Thus we have an isomorphism $\eta:A'\rightarrow A$, defined over a bigger extension $N\supseteq M'M$, such that $\imath(\alpha)\circ\eta=\eta\circ\imath'(\alpha)$, for all $\alpha\in\cO$. Since the isomorphism $\varphi$ is chosen in the $\bar\Q$-isomorphism class of $(A,\imath)$, its realization $\varphi'$ on $(A',\imath')$ satisfies $\varphi'=\varphi\circ\eta$. We compute
\[
\varphi(\mu_\sigma(P^\sigma))=\varphi'(\eta^{-1}\circ\mu_\sigma\circ\eta^{\sigma}((\eta^{-1}( P))^\sigma)).
\]
Since we have proved that $\rho_{(A',\imath',\varphi')}$ does not depend on the choice of the $\cO$-isogenies, we can choose $\eta^{-1}\circ\mu_\sigma\circ\eta^{\sigma}$ obtaining the desired result $\rho_{(A',\imath',\varphi')}=\rho_{(A,\imath,\varphi)}$.
\end{proof}

Note that $G(\A_f)/\End^0(A,\imath)^\times$ is a group in the non-CM case. Nevertheless, in the CM case, $\End(A,\imath)^0=K$ an imaginary quadratic field, hence $K^\times$ it is not normal in $G(\A_f)$. 

Given the embedding $\End^0(A,\imath)^\times\hookrightarrow G(\A_f)$ described above, 
let us denote by $N_A$ the normalizer of $\End(A,\imath)$ in $G(\A_f)$. Note that $N_A=G(\A_f)$, in the non-CM case. Moreover, if $(A,\imath)$ has CM by the imaginary quadratic field $K$, then the $\ell$ component of $N_A$ is $(N_A)_\ell=K_\ell^\times\cup jK_\ell^\times$, with $j^2\in\Q^\times$ and $jk=\bar kj$ for all $k\in K_\ell^\times$.
In any case $N_A/\End^0(A,\imath)^\times$ is now a group.
\begin{lemma}\label{comm}
The map $\rho_{(A,\imath,\varphi)}$ factors through 
\[
\rho_{(A,\imath,\varphi)}:\Gal(\bar \Q/L)\stackrel{\rho_{(A,\imath,\varphi)}^N}{\longrightarrow} N_A/\End^0(A,\imath)^\times\hookrightarrow G(\A_f)/\End^0(A,\imath)^\times
\]
Moreover, the map $\rho_{(A,\imath,\varphi)}^N$ is a group homomorphism. 
\end{lemma}
\begin{proof}
On the one side, for all $\sigma\in\Gal(\bar \Q/L)$ and $\lambda\in\End(A,\imath)$ we have $$\rho^\mu_{(A,\imath,\varphi)}(\sigma)\lambda^\ast\rho^\mu_{(A,\imath,\varphi)}(\sigma)^{-1}\in\End(A,\imath)^0.$$ 
Indeed,
\begin{eqnarray*}
(\deg\mu_\sigma)\varphi(P)\rho^\mu_{(A,\imath,\varphi)}(\sigma)\lambda^\ast\rho^\mu_{(A,\imath,\varphi)}(\sigma)^{-1}&=&(\deg\mu_\sigma)\varphi(\lambda(\mu_\sigma(P^\sigma)))\rho^\mu_{(A,\imath,\varphi)}(\sigma)^{-1}\\
&=&\varphi(\mu_\sigma^\vee(\lambda(\mu_\sigma(P^\sigma)))^{\sigma^{-1}})\\
&=&\varphi((\mu_\sigma^\vee\lambda\mu_\sigma)^{\sigma^{-1}}(P))\\
&=&\varphi(P)\left((\mu_\sigma^\vee\lambda\mu_\sigma)^{\sigma^{-1}}\right)^\ast,
\end{eqnarray*}
where clearly $(\mu_\sigma^\vee\lambda\mu_\sigma)^{\sigma^{-1}}\in\End^0(A,\imath)$. Therefore $\rho^\mu_{(A,\imath,\varphi)}(\sigma)\in N_A$.

On the other side, one checks that 
\[
\rho^\mu_{(A,\imath,\varphi)}(\sigma\tau)^{-1}\rho^\mu_{(A,\imath,\varphi)}(\sigma)\rho^\mu_{(A,\imath,\varphi)}(\tau)
\]
acts on $\hat TA\otimes\Q:=(\prod'_pT_pA)\otimes\Q$ in the same way as does
\[
c_{(A,\imath)}(\sigma,\tau) = (1/\deg(\mu_{\sigma\tau}))\mu_\sigma\mu_\tau^\sigma\mu_{\sigma\tau}^\vee\in(\End(A, \imath) \otimes_\Z \Q)^\times= \End^0(A,\imath)^\times.
\]
In particular, the quotient $\rho_{(A,\imath,\varphi)}(\sigma)$ is a group homomorphism.
\end{proof}

\begin{remark}
Assume that the discriminant $D=1$, thus the quaternion algebra $B=\M_2(\Q)$. An abelian surface with QM by $\cO=\M_2(\Z)$ is the product $A=E\times E$, where $E$ is an elliptic curve. In the particular case that $E$ is defined over $L$ (thus clearly $A=E\times E$ is an abelian $L$-surface with QM), the representation $\rho_{(A,\imath,\varphi)}$ is just the quotient modulo $\End^0(A,\imath)^\times=\End^0(E)^\times$ of the classical action on the Tate module
\[
\rho_E:\Gal(\bar \Q/L)\longrightarrow\GL_2(\hat\Z)=\prod_\ell\GL_2(\Z_\ell)\hookrightarrow\GL_2(\A_f).
\]
\end{remark}

\section{Shimura curves and isogeny classes}\label{sec3}

Assume that $\cO_0$ is a maximal order in $B$, let $\Gamma$ be an open subgroup of $\hat\cO_0^\times=G(\hat \Z)$. We say that two $\cO_0$-module isomorphisms $\varphi,\varphi':A_{tor}\stackrel{\simeq}{\rightarrow}B/\cO_0$ are \emph{$\Gamma$-equivalent} if there exists an element $\gamma\in \Gamma$ such that $\varphi'=\varphi\gamma$.
The \emph{Shimura curve} $X_\Gamma$, is the compactification of the coarse moduli space of triples $(A,\imath,\bar\varphi)$, where $(A,\imath)$ are abelian surfaces with QM by $\cO_0$ and $\bar\varphi$ is a $\Gamma$-equivalence class of $\cO_0$-module isomorphisms $\varphi:A_{tor}\stackrel{\simeq}{\rightarrow}B/\cO_0$. Such coarse moduli space is already compact unless $D=1$. The curve $X_\Gamma$ is defined over some number field $L_\Gamma$. If $k$ is a field of characteristic zero, 
given a point $P\in X_\Gamma(\bar k)$ corresponding to the isomorphism class of a triple $(A,\imath,\bar\varphi)/\bar k$, its Galois conjugate $P^\sigma\in X_\Gamma(\bar k)$, for any $\sigma\in\Gal(\bar k/k)$, corresponds to the isomorphism class of $(A^\sigma,\imath^\sigma,\bar\varphi^\sigma)$, where 
\[
\varphi^\sigma:A_{tor}^\sigma\stackrel{\simeq}{\longrightarrow}B/\cO_0;\qquad \varphi^\sigma(Q^\sigma)=\varphi(Q). 
\] 
Thus, a $k$-rational point $P$ in $X_\Gamma$ corresponds to the isomorphism class of a triple $(A,\imath,\bar\varphi)/\bar k$ which is isomorphic to all its $\Gal(\bar k/k)$-conjugates.

The complex points of the Shimura curve are in correspondence with the double coset space
\[
X_\Gamma(\C)=(\Gamma_\infty \Gamma\backslash G(\A))\slash G(\Q)\cup\{\mbox{cusps}\},\quad \Gamma_\infty=\left\{\left(\begin{array}{cc}a&b\\-b&a\end{array}\right)\in\GL_2(\R)\right\},
\] 
where cusps only appear in case $D=1$.
The triple $(A_g,\imath_g,\bar\varphi_g)$ over $\C$ corresponding to $g=(g_\infty, g_f)\in G(\A)$ is $A_{g}:=(B\otimes\R)_{g_\infty}/I_{g_f}$, where $I_{g_f}=\hat\cO_0 g_f\cap B$ and $(B\otimes\R)_{g_\infty}=\M_2(\R)$ with complex structure $h_{g_\infty}$
\[
h_{g_\infty}:\C\rightarrow\M_2(\R);\;i\mapsto g_\infty^{-1}\left(\begin{array}{cc}&1\\-1&\end{array}\right)g_\infty;
\]
the embedding $\imath_g:\cO_0\rightarrow\End(A_g)$, is given by $\imath_g(\alpha)(b\otimes z)=\alpha b\otimes z$; and $\bar\varphi_g$ is the $\Gamma$-equivalence class of $\varphi_g:(A_g)_{tor}=B/I_{g_f}\rightarrow B/\cO_0$, $\varphi_g(b)=b g_f^{-1}$.
We compute that
\begin{eqnarray*}
\End^0(A_{g},\imath_{g})^\times&=&\{\gamma\in \Aut_B(B\otimes\R):\;\gamma I_{g_f}\otimes\Q=I_{g_f}\otimes\Q\;\mbox{and}\;\gamma h_{g_\infty}=h_{g_\infty}\gamma\}\\
&=&\{\gamma\in G(\R):\;\gamma B=B\;\mbox{and}\;\gamma h_{g_\infty}\gamma^{-1}=h_{g_\infty}\}\\
&=&\{\gamma\in G(\Q):\;\gamma h_{g_\infty}\gamma^{-1}=h_{g_\infty}\}\\
&=&\{\gamma\in G(\Q):\;g_\infty\gamma g_\infty^{-1}\in\Gamma_\infty\}.
\end{eqnarray*}

\begin{remark}
In most of the literature, objects classified by the Shimura curve $X_\Gamma$ are triples $(A,\imath,\bar\psi)$, where $(A,\imath)$ is an abelian surface with QM by $\cO_0$ as above and $\bar\psi$ is a $\Gamma$-equivalence class of $\cO_0$-module isomorphisms $\psi:\hat T(A)=\Hom(A_{tor},\Q/\Z)\stackrel{\simeq}{\rightarrow}\hat\cO_0$. It is clear that this interpretation is equivalent to ours, since for any $\varphi:A_{tor}\stackrel{\simeq}{\rightarrow} B/\cO_0$ we have the corresponding isomorphism 
\[
\psi:\hat T(A)=\Hom(A_{tor},\Q/\Z)\stackrel{\simeq}{\longrightarrow}\Hom(B/\cO_0,\Q/\Z)\simeq\hat\cO_0.
\]
\end{remark}
\begin{remark}
In the particular case that $\Gamma=\Gamma_N=\ker(G(\hat \Z)\rightarrow G(\Z/N\Z))$, to give a $\Gamma$-equivalence class of isomorphisms $\varphi:A_{tor}\rightarrow B/\cO_0$ is equivalent to give an isomorphism $\varphi_N:A[N]\rightarrow\cO_0/N\cO_0$, namely, a level-$N$-structure. This is the classical Shimura curve situation.
\end{remark}

Let $\bar k$ be a field of characteristic 0 algebraically closed. We say that two triples $(A,\imath,\bar\varphi)$ and $(A',\imath',\bar\varphi')$ over $\bar k$ are isogenous if $(A,\imath)$ and $(A',\imath')$ are isogenous.

Let $P\in X_\Gamma(\C)$ be a point corresponding to $(A,\imath,\bar\varphi)$.
Let us denote by $[P]$ the $\C$-isogeny class of $(A,\imath,\bar\varphi)/\C$ in $X_\Gamma$, namely, the set of points $Q\in  X_\Gamma(\C)$ parametrizing triples $(A',\imath',\bar\varphi')/\C$ where $(A',\imath')$ is isogenous to $(A,\imath)$.
\begin{proposition}\label{propisoclas}
Let $P=[g]=[g_\infty,1]\in (\Gamma_\infty \Gamma\backslash G(\A))\slash G(\Q)\subseteq X_\Gamma(\C)$. Then we have the following bijection
\[
\psi_{g_\infty}:\Gamma\backslash G(\A_f)/\End^0(A_{g},\imath_{g})^\times\stackrel{\simeq}{\longrightarrow}[P];\qquad g_f\longmapsto [g_\infty,g_f].
\]
\end{proposition}
\begin{proof}
The non-CM case is described in \cite[Lemma 1]{M-G}, we give here a proof that works in any case. Recall that $(A_{g_\infty},\imath_{g_\infty},\bar\varphi_{g_\infty})$ is the triple corresponding to $P=[g_\infty,1]$. For any $g_f\in G(\A_f)$, there exists $n\in\Z$ such that $I_{g_f}n\subseteq\cO_0$. Therefore we have the isogeny
\[
A_{g_\infty g_f}=(B\otimes\R)_{g_\infty}/I_{g_f}\longrightarrow(B\otimes\R)_{g_\infty}/\cO_0= A_{g_\infty},\quad b\longmapsto nb,
\]
which is clearly a $\cO_0$-isogeny with respect to $\imath_{g_\infty}$ and $\imath_{g_\infty g_f}$ since the inclusion $I_{g_f}n\subseteq\cO_0$ is a monomorphism of $\cO_0$-modules. This implies $[g_\infty, g_f]\in [P]$, for all $g_f\in G(\A_f)$.

Conversely, any $\cO_0$-isogeny $(A_{g_\infty'g_f},\imath_{g_\infty'g_f})\rightarrow(A_{g_\infty},\imath_{g_\infty})$ induces an equality of complex structures $(B\otimes\R)_{g_\infty'}=(B\otimes\R)_{g_\infty}$. This implies that $g_\infty'=\Gamma_\infty g_\infty$. Therefore the corresponding point $[g_\infty',g_f]$ has a representant of the form $[g_\infty,g_f']$ in the double coset space $(\Gamma_\infty\Gamma\backslash G(\A))\slash G(\Q)$.

We conclude that the map
\[
\Gamma\backslash G(\A_f)\longrightarrow [P];\qquad g_f\longmapsto [g_\infty,g_f],
\]
is surjective. Finally, the result follows from the fact that $[g_\infty,g_f]=[g_\infty,g_f']$ in $(\Gamma_\infty\Gamma\backslash G(\A))\slash G(\Q)$ if and only if there exists $\beta\in G(\Q)$ such that $g_f=g_f'\beta$ and $g_\infty\beta\in\Gamma_\infty g_\infty$, hence $\beta\in\End^0(A_{g_\infty},\imath_{g_\infty})^\times$.
\end{proof}
\begin{remark}
The above proposition asserts that the isogeny class $[P]$ corresponds to the fiber containing $P$ of the natural map
\[
X_\Gamma\supseteq(\Gamma_\infty\Gamma\backslash G(\A))\slash G(\Q)\longrightarrow \Gamma_\infty\backslash G(\R)\slash G(\Q),\qquad [g_\infty,g_f]\longmapsto [g_\infty].
\]
\end{remark}

\section{Galois action on isogeny classes}\label{sec4}

Assume now that $(A,\imath)$ is an abelian $L$-surface with QM by $\cO_0$, let $(A,\imath,\bar\varphi)$ be a triple corresponding to the point $P\in X_\Gamma(\bar\Q)$. First we show that any $(A',\imath')$ isogenous to $(A,\imath)$ is an abelian $L$-surface with QM.
\begin{lemma}\label{isoclasses}
Let $(A,\imath)$ be an abelian $L$-surface with QM and assume that $(A',\imath')$ over $\bar\Q$ is isogenous to $(A,\imath)$. Then $(A',\imath')$ is an abelian $L$-surface with QM. 
\end{lemma}
\begin{proof}
Let $\sigma\in\Gal(\bar L/L)$. Since $(A,\imath)$ is an abelian $L$-surface with QM, there exists a $\cO_0$-isogeny $(A^\sigma,\imath^\sigma)\stackrel{\mu_\sigma}{\rightarrow}(A,\imath)$. Fix a $\cO_0$-isogeny $(A',\imath')\stackrel{\phi}{\rightarrow}(A,\imath)$  defined over $\bar\Q$ (such a $\cO_0$-isogeny exists since $(A,\imath)$ and $(A',\imath')$ are isogenous and both defined over $\bar\Q$). Thus by conjugating $\phi$ by $\sigma$ and composing with $\phi^\vee\circ\mu_\sigma$, one obtains
\[
((A')^\sigma,(\imath')^\sigma)\stackrel{\phi^\sigma}{\longrightarrow}(A^\sigma,\imath^\sigma)\stackrel{\mu_\sigma}{\longrightarrow}(A,\imath)\stackrel{\phi^\vee}{\longrightarrow}(A',\imath').
\]
Hence $(A',\imath')/\bar\Q$ is an abelian $L$-surface with QM.
\end{proof}
Note that, since $P$ and so $(A,\imath)$ are defined over $\bar\Q$, the $\C$-isogeny class coincide with the $\bar\Q$-isogeny class $[P]$.
Moreover, the above lemma implies that $\Gal(\bar \Q/L)$ acts on $[P]$. Indeed, if $Q\in[P]$ corresponds to $(A',\imath',\bar\varphi')$ and $\sigma\in\Gal(\bar \Q/L)$, then $Q^\sigma$ parametrizes $((A')^\sigma,(\imath')^\sigma,(\bar\varphi')^\sigma)$. Since $(A',\imath')$ is an L-abelian surface with QM by the lemma, there exists a $\cO_0$-isogeny
$\mu_\sigma':((A')^\sigma,(\imath')^\sigma)\rightarrow(A',\imath')$. This implies $((A')^\sigma,(\imath')^\sigma)$ is isogenous to $(A,\imath)$, hence $Q^\sigma\in[P]$. The main theorem of this section relates this action with the map $\rho_{(A,\imath,\varphi)}$ introduced in \S \ref{sec2} by means of the characterization of $[P]$ given in Proposition \ref{propisoclas}.
\begin{theorem}\label{TeoGalact}
Assume that $P=[g_\infty,1]\in X_\Gamma$ corresponds to a triple $(A,\imath,\bar\varphi)$, where $(A,\imath)/\bar\Q$ is an abelian $L$-surface with QM and $\bar\varphi$ is the $\Gamma$-equivalent class of the natural isomorphism
\[
\varphi:A_{tor}=((B\otimes\R)_{g_\infty}/\cO_0)_{tor}\longrightarrow B/\cO_0,
\] 
Then the map $\rho_{(A,\imath,\varphi)}: \Gal(\bar \Q/L) \longrightarrow G(\A_f)/\End^0(A,\imath)^\times$ constructed by means of $\varphi$ satisfies
\[
\psi_{g_\infty}([ g_f])^\sigma=\psi_{g_\infty}([g_f\rho_{(A,\imath,\varphi)}(\sigma)])\in [P],
\] 
for all $g_f\in G(\A_f)$ and $\sigma\in\Gal(\bar \Q/L)$, where $[\;\cdot\;]$ denothes the class in the double coset space $\Gamma\backslash G(\A_f)\slash\End^0(A,\imath)^\times$.
\end{theorem}
\begin{remark}
Note that, by Lemma \ref{comm}, the image $\rho_{(A,\imath,\varphi)}(\sigma)$ lies in the commutator of $\End^0(A,\imath)$ in $G(\A_f)$. Thus the product $g_f\rho_{(A,\imath,\varphi)}(\sigma)$ is well defined in $\Gamma\backslash G(\A_f)/\End^0(A,\imath)$.
\end{remark}
\begin{proof}
Recall that the abelian surface corresponding to $\psi_{g_\infty}([g_f])$ is given by the complex torus $A_{g_f}=(B\otimes\R)_{g_\infty}/I_{g_f}$, where $I_{g_f}=B\cap\hat\cO_0g_f$. Moreover, considering a representative of $[ g_f]$ such that $g_f^{-1}\in\hat\cO_0$, the $\cO_0$-isogeny between $(A,\imath)$ and $(A_{g_f},\imath_{g_f})$ is given by
\[
\phi_{g_f}:A=(B\otimes\R)_{g_\infty}/\cO_0\longrightarrow (B\otimes\R)_{g_\infty}/I_{g_f}=A_{g_f},\quad b\longmapsto b.
\]
Also recall that a representative of $\bar\varphi_{g_f}$ is given by
\[
\varphi_{g_f}:(A_{g_f})_{tor}=((B\otimes\R)_{g_\infty}/I_{g_f})_{tor}=B/I_{g_f}\longrightarrow B/\cO_0,\quad b\longmapsto bg_f^{-1}.
\] 
Thus one checks that 
\begin{equation}\label{characgf}
\varphi_{g_f}\circ\phi_{g_f}=\varphi g_f^{-1}:A_{tor}\rightarrow B/\cO_0.
\end{equation}

For any $\sigma\in\Gal(\bar \Q/L)$, the point $\psi_{g_\infty}([g_f])^\sigma$ corresponds to the triple $(A_{g_f}^\sigma,\imath_{g_f}^\sigma,\bar\varphi_{g_f}^{\sigma})$. We have the following $\cO_0$-isogenies
\[
(A_{g_f}^\sigma,\imath_{g_f}^\sigma)\stackrel{\phi_{g_f}^\sigma}{\longleftarrow}(A^\sigma,\imath^\sigma)\stackrel{\mu_\sigma}{\longrightarrow}(A,\imath)\stackrel{\phi_{g_f}}{\longrightarrow} ( A_{g_f},\imath_{g_f}),
\]
thus $(A_{g_f}^\sigma,\imath_{g_f}^\sigma)$ and $(A,\imath)$ are linked by the $\cO_0$-isogeny $\phi_{g_f}^\sigma\circ\mu_\sigma^\vee$.
This implies that, as in the case of equation \eqref{characgf}, we have a representative $g_f^\sigma\in G(\A_f)$ of the double coset $\psi_{g_\infty}^{-1}(\psi_{g_\infty}([g_f])^\sigma)\in\Gamma\backslash G(\A_f)/\End^0(A,\imath)$ satisfying $\varphi_{g_f}^\sigma\circ(\phi_{g_f}^\sigma\circ\mu_\sigma^\vee)=\varphi(g_f^\sigma)^{-1}$. Hence, for all $P\in A_{tor}$,
\begin{eqnarray*}
\varphi(P)\rho_{(A,\imath,\varphi)}(\sigma)(g_f^\sigma)^{-1}&=&\varphi(\mu_\sigma(P^\sigma))(g_f^\sigma)^{-1}=\deg(\mu_\sigma)\varphi_{g_f}^\sigma(\phi_{g_f}^\sigma(P^\sigma))\\
&=&\deg(\mu_\sigma)\varphi_{g_f}^\sigma(\phi_{g_f}(P)^\sigma)=\deg(\mu_\sigma)\varphi_{g_f}(\phi_{g_f}(P))\\
&=&\deg(\mu_\sigma)\varphi(P)g_f^{-1}
\end{eqnarray*}
We conclude $[g_f\rho_{(A,\imath,\varphi)}(\sigma)]=[g_f^\sigma]=\psi_{g_\infty}^{-1}(\psi_{g_\infty}([g_f])^\sigma)$ and the result follows.
\end{proof}

\section{Change of moduli interpretation}\label{sec5}

In section \S\ref{sec2}, we defined an abelian $L$-surface $(A,\imath)$ with QM by any Eichler order $\cO$ and defined the corresponding representation $\rho_{(A,\imath,\varphi)}$ attached to a fixed $\cO$-module isomorphism $\varphi:A_{tor}\rightarrow B/\cO$. Nevertheless, we used a maximal order $\cO_0$ to define the Shimura curve $X_\Gamma$ and to describe its moduli interpretation as the space classifying triples $(A_0,\imath_0,\bar\varphi_0)$, where $(A_0,\imath_0)$ has QM by $\cO_0$. In this section we shall change this moduli interpretation for some of this Shimura curves $X_\Gamma$ in order to classify abelian surfaces with QM by $\cO$.

Thus from now on $\cO$ will be an Eichler order in $B$ of level $N$ and $\cO_0$ a maximal order such that $\cO\subseteq\cO_0$. Fix the embedding $\lambda:\cO\hookrightarrow\cO_0$. Let $\Gamma$ be now an open subgroup of $\hat\cO^\times=(\cO\otimes\hat\Z)^\times$. Since $\hat\cO^\times$ is an open subset of $G(\hat\Z)=\hat\cO_0^\times$ by means of $\lambda$, the subgroup $\Gamma$ is also an open subgroup of $G(\hat\Z)$. Thus we can consider the Shimura curve $X_\Gamma$.

\begin{proposition}\label{propinterp}
We have an equivalence of moduli interpretations for the Shimura curve $X_\Gamma$. It either classifies:
\begin{itemize}
\item[(i)] Triples $(A_0,\imath_0,\bar\varphi_0)$, where $(A_0,\imath_0)$ is an abelian surface with QM by $\cO_0$ and $\bar\varphi_0$ is a $\Gamma$-equivalence class of $\cO_0$-module isomorphisms $\varphi_0:(A_0)_{tor}\rightarrow B/\cO_0$. 

\item[(ii)] Triples $(A,\imath,\bar\varphi)$, where $(A,\imath)$ is an abelian surface with QM by $\cO$ and $\bar\varphi$ is a $\Gamma$-equivalence class of $\cO$-module isomorphisms $\varphi:A_{tor}\rightarrow B/\cO$. 
\end{itemize}
\end{proposition}
In order to prove this proposition we will need the following lemma. Note that the embedding $\lambda:\cO\hookrightarrow\cO_0$ gives rise to a morphism $\lambda:B/\cO\rightarrow B/\cO_0$.
\begin{lemma}\label{leminterp}
There exists a one-to-one correspondence between triples $(A,\imath,\varphi)$, where $(A,\imath)$ is an abelian surface with QM by $\cO$ and $\varphi$ is a $\cO$-module isomorphism $\varphi:A_{tor}\rightarrow B/\cO$, and triples $(A_0,\imath_0,\varphi_0)$, where $(A_0,\imath_0)$ is an abelian surface  with QM by $\cO_0$ and $\varphi_0$ is a $\cO_0$-module isomorphism $\varphi_0:(A_0)_{tor}\rightarrow B/\cO_0$. A triple $(A,\imath,\varphi)$ corresponds to $(A_0,\imath_0,\varphi_0)$ if there exists an isogeny $\phi:A\rightarrow A_0$, such that 
$\varphi_0\circ\phi=\lambda\circ\varphi$ and $\phi\circ\imath(\alpha)=\imath_0(\lambda(\alpha))\circ\phi$, for all $\alpha\in\cO$. 
\end{lemma}
\begin{proof}
Given $(A,\imath,\varphi)$, let us consider the subgroup $C:=\varphi^{-1}(\ker(B/\cO\stackrel{\lambda}{\rightarrow}B/\cO_0))\subset A_{tor}$. Therefore, we can construct the abelian surface $A_0=A/C$ and the corresponding isogeny $\phi:A\rightarrow A_0$. Since $\cO\subseteq\cO_0$, for all $\alpha\in\cO$, we have $\alpha(\ker\lambda)\subseteq\ker\lambda$, hence, $\imath(\alpha)C\subseteq C$ and the embedding $\imath$ gives rise to an embedding $\imath_0:\cO\hookrightarrow\End(A_0)$. The $\cO$-module isomorphim $\varphi$ gives rise to a $\cO$-module isomorphism $\varphi_0$ that fits into the following commutative diagram:
\[
\xymatrix{
A_{tor}\ar[r]^{\varphi}\ar[d]_{\phi}&B/\cO\ar[d]^\lambda\\
(A_0)_{tor}=A_{tor}/C\ar[r]^{\qquad\varphi_0}&B/\cO_0
}
\] 
Hence $\varphi_0\circ\phi=\lambda\circ\varphi$. Moreover, the fact that $(A_0)_{tor}\simeq B/\cO_0$ as $\cO$-modules implies that $\imath_0$ can be extended to an embedding $\imath_0:\cO_0\hookrightarrow\End(A_0)$. Thus $(A_0,\imath_0)$ has QM by $\cO_0$. We have constructed the triple $(A_0,\imath_0,\varphi_0)$ corresponding to $(A,\imath,\varphi)$.

Finally, given $(A_0,\imath_0,\varphi_0)$, let us consider $C^\vee:=\varphi_0^{-1}(\ker(B/\cO_0\stackrel{\lambda^\vee}{\rightarrow}B/\cO))$, where $\lambda^\vee:B/\cO_0\rightarrow B/\cO$ is the well defined morphism $\lambda^\vee(b+\cO_0)=[\cO_0:\cO]b+\cO$.
We define $A:=A_0/C^\vee$. Notice that, for all $o\in\cO\subset\cO_0$, we have $\imath_0(o)(C^\vee)\subseteq C^\vee$. Hence $\imath_0$ gives rise to an embedding $\imath:\cO\hookrightarrow\End(A)$ and $\varphi_0$ provides a $\cO$-module isomorphism $\varphi:A_{tor}\rightarrow B/\cO$. It is easy to check that this construction is the inverse to the previous one, thus the result follows.
\end{proof}

Due to this previous lemma we can easily prove the above proposition:
\begin{proof}[Proof of Proposition \ref{propinterp}]
We know that the Shimura curve $X_\Gamma$ classify triples $(A_0,\imath_0,\bar\varphi_0)$ as in $(i)$. By the above Lemma, given a representative $\varphi_0$ of the $\Gamma$-equivalence class $\bar\varphi_0$, there exists a triple $(A,\imath,\varphi)$, where $\varphi$ is a $\cO$-module isomorphism. It is clear that the $\Gamma$-equivalence class $\bar\varphi_0$ corresponds to the $\Gamma$-equivalence class $\bar\varphi$.
\end{proof}

\begin{definition}
A \emph{triple with QM by $\cO$} is a triple $(A,\imath,\varphi)$, where $(A,\imath)$ is an abelian surface with QM by $\cO$ and $\varphi$ is a $\cO$-module isomorphism $\varphi:A_{tor}\rightarrow B/\cO$. A \emph{$L$-triple with QM by $\cO$} is a triple $(A,\imath,\varphi)$ with QM by $\cO$ such that $(A,\imath)$ is an abelian $L$-surface with QM.
\end{definition}

We denote the one-to-one correspondence of Lemma \ref{leminterp} by
\[
\Lambda^{\cO_0}_\cO:\left\{\mbox{Triples with QM by $\cO$}\right\}\longrightarrow\left\{\mbox{Triples with QM by $\cO_0$}\right\}
\]

Note that, given a $L$-triple $(A,\imath,\varphi)$ with QM by $\cO$, one can construct the projective representation 
\[
\rho_{(A,\imath,\varphi)}:\Gal(\bar L/L)\longrightarrow G(\A_f)/\End^0(A,\imath)^\times.
\]  
The following result relates the representations attached to triples associated by the correspondence $\Lambda^{\cO_0}_\cO$.

\begin{lemma}
Let $(A,\imath,\varphi)$ be a $L$-triple with QM by $\cO$.
Assume that $\Lambda^{\cO_0}_\cO(A,\imath,\varphi)=(A_0,\imath_0,\varphi_0)$, then $(A_0,\imath_0,\varphi_0)$ is a $L$-triple with QM by $\cO_0$ and
\[
\rho_{(A,\imath,\varphi)}=\rho_{(A_0,\imath_0,\varphi_0)},
\]
when we identify $G(\A_f)=\End^0_\cO(B/\cO)^\times=\End^0_{\cO_0}(B/\cO_0)^\times$ by means of $\lambda$.
\end{lemma}
\begin{proof}
We know that there exists an isogeny $\phi:A\rightarrow A_0$, such that 
$\varphi_0\circ\phi=\lambda\circ\varphi$ and $\phi\circ\imath(\alpha)=\imath_0(\lambda(\alpha))\circ\phi$, for all $\alpha\in\cO$. Since $(A,\imath)$ is an abelian $L$-surface with QM, there exists a set of $\cO$-isogenies $\mu=\{\mu_\sigma:(A^\sigma,\imath^\sigma)\rightarrow (A,\imath),\;\sigma\in\Gal(\bar \Q/L)\}$. The composition
\[
\mu_\sigma^0:A_0^\sigma\stackrel{(\phi^\sigma)^\vee}{\longrightarrow}A^\sigma\stackrel{\mu_\sigma}{\longrightarrow}A\stackrel{\phi}{\longrightarrow}A_0,
\]
satisfies 
\begin{eqnarray*}
\mu_\sigma^0\circ\imath_0(\lambda(o))^\sigma&=&\phi\circ\mu_\sigma\circ(\phi^\sigma)^\vee\circ\imath_0(\lambda(o))^\sigma=\phi\circ\mu_\sigma\circ\imath(o)^\sigma\circ(\phi^\sigma)^\vee\\
&=&\phi\circ\imath(o)\circ\mu_\sigma\circ(\phi^\sigma)^\vee=\imath_0(\lambda(o))\circ\mu_\sigma^0,\; \mbox{ for all }o\in\cO,
\end{eqnarray*}
thus $\mu_\sigma^0\circ\imath_0(\alpha)^\sigma=\imath_0(\alpha)\circ\mu_\sigma^0$ for all $\alpha\in\cO_0$. This implies that $(A_0,\imath_0)$ is an abelian $L$-surface with QM.

Moreover, for all $\sigma\in\Gal(\bar \Q/L)$ and $P\in (A_0)_{tor}$, we have
\begin{eqnarray*}
\varphi_0(\mu_\sigma^0(P^\sigma))&=&\varphi_0(\phi\circ\mu_\sigma\circ(\phi^\sigma)^\vee(P^\sigma))=\lambda(\varphi(\mu_\sigma(\phi^\vee(P)^\sigma)))\\
&=&\lambda(\varphi(\phi^\vee(P)))\rho^\mu_{(A,\imath,\varphi)}(\sigma)=\deg(\phi)\varphi_0(P)\rho^\mu_{(A,\imath,\varphi)}(\sigma).
\end{eqnarray*}
This implies that $\rho_{(A,\imath,\varphi)}(\sigma)=\rho_{(A_0,\imath_0,\varphi_0)}(\sigma)$.
\end{proof}

\begin{remark}
As a consequence of this lemma, we obtain that Theorem \ref{TeoGalact} also applies if we change the maximal order $\cO_0$ by a not necessarily maximal Eichler order $\cO$, considering the moduli interpretation $(ii)$ of Proposition \ref{propinterp}.
\end{remark}

\section{Duality}\label{sec6}

In this section we describe the dual of an abelian surface with quaternionic multiplication and its associated Weil pairing. 

Let $(A,\imath)/\C$ be an abelian surface with QM by $\cO$. Denote by $A^\vee/\C$ its dual abelian surface. Denote by $\langle\;,\;\rangle$ the Weil pairing 
\[
\langle\;,\;\rangle:A_{tor}\times A^\vee_{tor}\longrightarrow \{\zeta_n, \;n\in\N\},
\]
where $\{\zeta_n, n\in\N\}$ is the group of roots of unity. Such group is isomorphic to $\Q/\Z$ by means of the isomorphism
\[
\psi:\{\zeta_n, \;n\in\N\}\longrightarrow\Q/\Z;\quad e^{\frac{2\pi i}{n}m}\longmapsto m/n.
\]
Given $\cO$, we have the two-sided ideal
\[
\cO^\#=\{b\in B:\;\Trace(b\cO)\subseteq\Z\},
\]
where $\Trace$ denotes the reduced trace. Given any left $\cO$-ideal $I$, we denote by ${\rm Norm}(I)$ its reduced norm.

\begin{proposition}
The dual abelian surface $A^\vee$ admits a quaternionic multiplication $\imath^\vee$ such that $(A^\vee,\imath^\vee)$ is isogenous to $(A,\imath)$ by means of a $\cO$-isogeny $\varepsilon:(A,\imath)\rightarrow (A^\vee,\imath^\vee)$ of degree ${\rm Norm}(\cO^\#)^{-1}$. The dual isogeny $\varepsilon^\vee$ satisfies $\varepsilon^\vee=\varepsilon$ and the quaternionic multiplication $\imath^\vee$ satisfies $\imath^\vee(\alpha)=\imath(\bar\alpha)^\vee$, for all $\alpha\in\cO$. Moreover, for any $\cO$-module isomorphism $\varphi:A_{tor}\rightarrow B/\cO$, there exist $u\in \hat\Z^\times$ such that the Weil pairing satisfies
\[
\psi(\langle P,\varepsilon(Q)\rangle)=u\Trace(\varphi(P)\overline{\varphi(Q)}), 
\]
for any $P,Q\in A_{tor}$.
\end{proposition}
\begin{proof}
We have seen that, as a complex torus, $A=(B\otimes\R)_h/I$, for some left $\cO$-ideal $I$ and some complex structure $h:\C\rightarrow B\otimes\R$. Its dual complex torus corresponds to $(B\otimes\R)^\vee_{h^\vee}/I^\vee$, where 
\[
(B\otimes\R)^\vee:=\{f:(B\otimes\R)\rightarrow\C,\;\R{\rm -linear},\;f(h(i)v)=-if(v)\},
\]\[
 I^\vee:=\{f\in (B\otimes\R)^\vee:\;{\rm Im}f(I)\subseteq\Z\},
\]
and the complex structure 
$h^\vee:\C^\times\longrightarrow \Aut_\R((B\otimes\R)^\vee)$ is given by $h^\vee(z)f=zf$. The non-degenerate pairing
\[
B\times B\longrightarrow \Q;\quad (b_1,b_2)\longmapsto \Trace(b_2\bar b_1),
\]
provides an isomorphism between $(B\otimes\R)^\vee_{h^\vee}$ and $(B\otimes\R)_h$
since $ \Trace(b_2h(i) \overline{h(i)}\bar b_1)= \Trace(b_2\bar b_1)$. Hence we obtain that 
\[
A^\vee\simeq(B\otimes\R)^\vee_{h^\vee}/I^\vee\simeq(B\otimes\R)_h/I^\#,\quad I^\#=\{b\in B:\;\Trace(b\bar I)\subseteq\Z\}.
\]
Since $I^\#$ is a left $\cO$-ideal, we deduce that $A^\vee$ admits a quaternionic multiplication $\imath^\vee$, and $(A^\vee,\imath^\vee)$ is in the $\cO$-isogeny class of $(A,\imath)$. A $\cO$-isogeny $\varepsilon:A\rightarrow A^\vee$ is given by the inclusion $\frac{1}{{\rm Norm}(I)}I\subseteq I^\#$. Its kernel corresponds to the quotient $I^\#/(\frac{1}{{\rm Norm}(I)}I)\simeq \cO^\#/\cO$, hence we obtain that $\deg(\varepsilon)={\rm Norm}(\cO^\#)^{-1}$. The isomorphism $(B\otimes\R)^\vee_{h^\vee}/I^\vee\simeq(B\otimes\R)_h/I^\#$, provides the following description of the Weil pairing:
\[
\begin{array}{ccl}
\langle\cdot,\cdot\rangle:A_{tor}\simeq B/I\times B/I^\#\simeq A^\vee_{tor}&\longrightarrow& \Q/\Z\stackrel{\psi}{\simeq}\{\zeta_n,\;n\in\N\}\\
(b,b')&\longmapsto&\Trace(b\bar b').
\end{array}
\]
Any isomorphism $\varphi:A_{tor}\rightarrow B/\cO$, is given by an element $g_f\in G(\A_f)$, such that $\hat\cO g_f\cap B=I$, and the composition 
\[
\varphi:A_{tor}\simeq B/I\stackrel{\simeq}{\longrightarrow}B/\cO;\quad b\longmapsto bg_f^{-1}.
\]
We compute, for $P,Q\in A_{tor}$ corresponding to $b,b'\in B/I$ respectivelly,
\[
\psi(\langle P,\varepsilon(Q)\rangle)=\Trace\left(b\frac{1}{{\rm Norm}(I)}\bar b'\right)=\frac{1}{{\rm Norm}(I)}\Trace(\varphi(P)g_f\overline{g_f}\overline{\varphi(Q)})=
u\Trace(\varphi(P)\overline{\varphi(Q)}),
\]
where $u={\rm Norm}(g_f)/{\rm Norm}(I)\in\hat\Z^{\times}$. The fact that $\varepsilon^\vee=\varepsilon$ follows from the above equality since, by definition, $\langle Q,\varepsilon^\vee(P)\rangle=\langle P,\varepsilon(Q)\rangle$ and the pairing given by the trace is symmetric. Finally, the identity $\imath^\vee(\alpha)=\imath(\bar\alpha)^\vee$ follows directly from the above description of the Weil pairing.
\end{proof}

\begin{remark}
If $\mu:(A_0,\iota_0)\rightarrow (A,\iota)$ is a $\cO$-isogeny, then we have 
\[
\mu^\vee\circ\iota^\vee(\alpha)=\mu^\vee\circ\iota(\bar\alpha)^\vee=(\iota(\bar\alpha)\circ\mu)^\vee=(\mu\circ\iota_0(\bar\alpha))^\vee=\iota_0(\bar\alpha)^\vee\circ\mu^\vee=\iota_0^\vee(\alpha)\circ\mu^\vee,
\]
hence $\mu^\vee:(A^\vee,\iota^\vee)\rightarrow (A_0^\vee,\iota_0^\vee)$ is a $\cO$-isogeny.
\end{remark}

\subsection{Weil pairing on abelian $L$-surfaces with QM}\label{WeiPair}

Let $(A,\iota)$ be an abelian $L$-surface with QM. Fix a model $(A,\imath)$ over some number field $M$, a set of $\cO$-isogenies $\mu=\{\mu_\sigma:(A^\sigma,\iota^\sigma)\rightarrow(A,\iota)\}$ defined over $M$ and a $\cO$-isomorphism $\varphi:A_{tor}\rightarrow B/\cO$. In \S\ref{sec2} we constructed the endomorphisms $\rho^\mu_{(A,\iota,\varphi)}(\sigma)\in G(\A_f)$, where $\sigma\in\Gal(\bar\Q/L)$, satisfying $\varphi(\mu_\sigma(P^\sigma))=\varphi(P)\rho^\mu_{(A,\iota,\varphi)}(\sigma)$, for all $P\in A_{tor}$. 

Since $(A^\vee,\imath^\vee)$ is $\cO$-isogenous to $(A,\imath)$, we know that $(A^\vee,\imath^\vee)$ is also an abelian $L$-surface with QM. Moreover, by the functorial description of the dual abelian surface, $(A^\vee,\imath^\vee)$ admits a model defined over $M$ satisfying $(A^\sigma)^\vee=(A^\vee)^\sigma$ and $\langle P^\sigma,Q^\sigma\rangle=\langle P,Q\rangle^\sigma$, for all $P\in A_{tor},\;Q\in A^\vee_{tor}$ and $\sigma\in\Gal(\bar\Q/\Q)$.

\begin{proposition}
We have that $\varepsilon:(A,\imath)\longrightarrow(A^\vee,\imath^\vee)$ is defined over $M$, moreover,
\[
\mu_\sigma^\vee\circ\varepsilon\circ\mu_\sigma=\deg(\mu_\sigma)\varepsilon^\sigma
\]
for all $\sigma\in\Gal(M/L)$.
\end{proposition}
\begin{proof}
First we consider a Galois extension $M'/M$ big enough such that $\varepsilon$ is defined over $M'$ and we extend the set of $\cO$-isogenies in the natural way: for $\sigma\in\Gal(M'/L)$, we define $\mu_\sigma:=\mu_{\pi(\sigma)}$, where $\pi:\Gal(M'/L)\rightarrow\Gal(M/L)$ is the natural projection.
We consider the Weil restriction of scalars $X=\Res_L A$. It is an abelian variety defined over $L$ and isomorphic over $\bar\Q$ to $X\simeq_{\bar\Q}\prod_{\sigma\in\Gal(M'/L)} A^\sigma$. For any $\tau\in\Gal(M'/L)$, the isogeny $\mu_\tau$ gives rise to an endomorphism $\lambda_\tau\in\End(X)$ defined by:
\[
\lambda_\tau:\prod_\sigma A^\sigma\longrightarrow\prod_\sigma A^\sigma,\quad(P_\sigma)_\sigma\longmapsto (\mu_\tau^{\sigma}(P_{\sigma\tau}))_{\sigma}
\]
Note that we have the polarization 
\[
\underline\varepsilon:=(\varepsilon^\sigma):X=\prod_\sigma A^\sigma\longrightarrow\prod_\sigma (A^\vee)^\sigma=X.
\]
Hence the Rosati involution of $\lambda_\tau$ with respect to $\underline\varepsilon$ is given by 
\begin{eqnarray*}
\lambda_\tau^\dagger(P_\sigma)_\sigma&=&\underline\varepsilon^{-1}\circ\lambda_\tau^\vee\circ\underline\varepsilon(P_\sigma)_\sigma=\underline\varepsilon^{-1}\circ\lambda_\tau^\vee(\varepsilon^\sigma(P_\sigma))_\sigma\\
&=&\underline\varepsilon^{-1}((\mu_\tau^\vee)^{\sigma}(\varepsilon^\sigma(P_\sigma))_{\sigma\tau})=((\varepsilon^{\sigma\tau})^{-1}((\mu_\tau^\vee)^{\sigma}(\varepsilon^\sigma(P_\sigma))))_{\sigma\tau}\\
&=&(((\varepsilon^{\tau})^{-1}\circ\mu_\tau^\vee\circ\varepsilon)^\sigma(P_\sigma))_{\sigma\tau}.
\end{eqnarray*}
Thus 
%\[
%\lambda_\tau\circ\lambda_\tau^\dagger(P_\sigma)_\sigma=((\mu_\tau\circ(\varepsilon^{\tau})^{-1}\circ\mu_\tau^\vee\circ\varepsilon)^\sigma(P_\sigma))_{\sigma},%=(\mu_\tau\circ(\varepsilon^{\tau})^{-1}\circ\mu_\tau^\vee\circ\varepsilon)(P_\sigma)_{\sigma},
%\]
$\lambda_\tau\circ\lambda_\tau^\dagger$ acts as the diagonal matrix with entries $(\mu_\tau\circ(\varepsilon^{\tau})^{-1}\circ\mu_\tau^\vee\circ\varepsilon)^\sigma\in\End(A^\sigma,\imath^\sigma)$.
Due to the fact that the Rosati involution is positive, we have $\mu_\tau\circ(\varepsilon^{\tau})^{-1}\circ\mu_\tau^\vee\circ\varepsilon\in\Q^{>0}$. Since $\deg((\varepsilon^{\tau})^{-1}\circ\mu_\tau^\vee\circ\varepsilon)=\deg(\mu_\tau)$, We conclude that $(\varepsilon^{\tau})^{-1}\circ\mu_\tau^\vee\circ\varepsilon=\hat\mu_\tau$, hence $\mu_\sigma^\vee\circ\varepsilon\circ\mu_\sigma=\deg(\mu_\sigma)\varepsilon^\sigma$. Since this equality is also true for any $\sigma\in\Gal(M'/M)$ where $\mu_\sigma={\rm Id}$, we also deduce that $\varepsilon$ is defined over $M$.
\end{proof}

Let $\chi:\Gal(\bar\Q/\Q)\rightarrow\hat\Z$ be the cyclotomic character that provides the Galois action on $\{\zeta_n,\;n\in\N\}\simeq \Q/\Z$, namely $\psi(\zeta_n^\sigma)=\chi(\sigma)\psi(\zeta_n)$, for any root of unity $\zeta_n$ and $\sigma\in\Gal(\bar\Q/\Q)$.

\begin{theorem}\label{normrho}
We have that
$$
{\rm Norm}(\rho^\mu_{(A,\imath,\varphi)}(\sigma))=\deg(\mu_\sigma)\chi(\sigma),
$$
for all $\sigma\in\Gal(\bar\Q/L)$.
\end{theorem}
\begin{proof}
We have seen that there exist $u\in\A_f^\times$ such that $\psi(\langle P,\varepsilon(Q)\rangle)=u\Trace(\varphi(P)\overline{\varphi(Q)})$. Thus we compute, for all $P,Q\in A_{tor}$,
\begin{eqnarray*}
\deg(\mu_\sigma)\chi(\sigma)\psi(\langle P,\varepsilon(Q)\rangle)&=&\deg(\mu_\sigma)\psi(\langle P,\varepsilon(Q)\rangle^\sigma)=\psi(\langle P^\sigma,\deg(\mu_\sigma)\varepsilon^\sigma(Q^\sigma)\rangle)\\
&=&\psi(\langle P^\sigma,\mu_\sigma^\vee\circ\varepsilon\circ\mu_\sigma(Q^\sigma)\rangle)=\psi(\langle \mu_\sigma(P^\sigma),\varepsilon(\mu_\sigma(Q^\sigma))\rangle)\\
&=&u\Trace(\varphi(\mu_\sigma(P^\sigma))\overline{\varphi(\mu_\sigma(Q^\sigma))})\\
&=&u\Trace(\varphi(Q)\;\rho^\mu_{(A,\imath,\varphi)}(\sigma)\;\overline{\rho^\mu_{(A,\imath,\varphi)}(\sigma)}\;\overline{\varphi( P))})\\
&=&{\rm Norm}(\rho^\mu_{(A,\iota,\varphi)}(\sigma))u\Trace(\varphi( P)\overline{\varphi(Q)})\\
&=&{\rm Norm}(\rho^\mu_{(A,\iota,\varphi)}(\sigma))\psi(\langle P,\varepsilon(Q)\rangle),
\end{eqnarray*}
and the result follows.
\end{proof}

\begin{corollary}
The composition
\[
\Gal(\bar\Q/L)\stackrel{\rho_{(A,\imath,\varphi)}}{\longrightarrow}G(\A_f)/\End^0(A,\imath)^\times\stackrel{{\rm Norm}}{\longrightarrow}\A_f^\times/\Q^{>0}\simeq \A_\Q^\times/\Q^\times\R^{>0}
\]
factors through $\Gal(\Q^{ab}/(L\cap\Q^{ab}))$ and it is given by the restriction of the inverse of the Artin map.
\end{corollary}

\section{Complex Multiplication abelian $K$-surfaces with QM}\label{sec7}

In this section we shall deal with the Complex Multiplication (CM) case.
Hence, only for this section, we assume that the abelian surface $(A,\imath)$ with QM by $\cO$ has also CM by $K$, so is to say, $\End^0(A,\imath)=K$ an imaginary quadratic field. The following result describes the projective representation $\rho^N_{(A,\imath,\varphi)}$ (and therefore $\rho_{(A,\imath,\varphi)}$) in the CM case:
\begin{proposition}\label{PropCM}
Assume that $\End^0(A,\imath)=K$ an imaginary quadratic field and let $\A_{K,f}$ be the ring of finite adeles of $K$. We have the following results
\begin{itemize}
\item[(i)] Any abelian surface $(A',\imath')$ with QM by $\cO$ and CM by $K$ is isogenous to $(A,\imath)$. %Moreover, there exists finitely many such a pairs $(A',\imath')$ satisfying $\End(A',\imath')=\cO_K$.

\item[(ii)] We can chose a representative of the isomorphism class of $(A,\imath)$ defined over $\bar\Q$. Moreover, $(A,\imath)$ is an abelian $\Q$-surface with QM.

\item[(iii)] Given any isomorphism $\varphi:A_{tor}\rightarrow B/\cO$, the restriction to $\Gal(\bar \Q/K)$ of the corresponding projective representation $\rho^N_{(A,\imath,\varphi)}$  factors through the inverse of the Artin map
${\rm Art}:\A_{K,f}^\times/K^\times\longrightarrow\Gal(K^{ab}/K)$, namely, we have the following commutative diagram
\[
\xymatrix{
\Gal(\bar\Q/\Q)\ar[rr]^{\rho^N_{(A,\imath,\varphi)}}&&N_A/K^\times\\
\Gal(\bar \Q/K)\ar@{^(->}[u]\ar@{->>}[r]&\Gal(K^{ab}/K)\ar[r]^{{\rm Art}^{-1}}&\A_{K,f}^\times/K^\times.\ar@{^(->}[u]
}
\]

\item[(iv)] The representation $\rho^N_{(A,\imath,\varphi)}$ factors through $\Gal(K^{ab}/K)\rtimes\Gal(K/\Q)$ sending the complex conjugation $\sigma_c=[1,\sigma_0]$, where $\langle\sigma_0\rangle=\Gal(K/\Q)$, to the class $jK^\times$, where $j\in N_A\cap B^\times$ is any element satisfying $j^2\in\Q^\times$ and $jk=\bar kj$ for all $k\in \A_{K,f}^\times$.
\end{itemize}
\end{proposition}
\begin{proof}
Let us consider first the open subgroup $\Gamma=\hat\cO^\times$.
We showed in \S \ref{sec5} that abelian surfaces with QM by $\cO$ over $\C$ are classified by the non-cuspidal points of $X_{\hat\cO^\times}$. Assume that $(A,\imath)$ with CM by $K$ corresponds to $[g_\infty,g_f]\in(\Gamma_\infty \hat\cO^\times\backslash G(\A))\slash G(\Q)\subseteq X_{\hat\cO^\times}(\C)$. The natural composition
\[
K^\times=\End^0(A,\imath)^\times=\{\gamma\in G(\Q):g_\infty\gamma g_\infty^{-1}\in\Gamma_\infty\}\hookrightarrow G(\Q)=B^\times,
\]
provides an embedding $\psi_{(A,\imath)}:K\hookrightarrow B$.
Given another $(A',\imath')$ with CM by K corresponding to $[g_\infty',g_f']$, we obtain an analogous embedding $\psi_{(A',\imath')}:K\hookrightarrow B$. By Skolem-Noether, there exists $g\in G(\Q)$, such that $\psi_{(A',\imath')}=g^{-1}\psi_{(A,\imath)}g$. This implies that $g_\infty'=g_\infty g$, hence $[g_\infty',g_f']=[g_\infty,g_f'g^{-1}]$. We conclude that $(A',\imath')$ is isogenous to $(A,\imath)$ by Proposition \ref{propisoclas}, thus part $(i)$ follows.

By Shimura's Reciprocity Law \cite[Main Theorem I]{Shi}, all points in $[P]$ are defined over $K^{ab}$. In particular $(A,\imath)$ is defined over $\bar\Q$. Since $(A^\sigma,\imath^\sigma)$ has CM by $K$ for any $\sigma\in\Gal(\bar\Q/\Q)$, we apply part $(i)$ to deduce that $(A,\imath)$ and $(A^\sigma,\imath^\sigma)$ are $\cO$-isogenous. We have proved part $(ii)$.

Additionally, Shimura's Reciprocity Law also describes the Galois action on the isogeny class of $(A,\imath)$. More precisely, for any open subgroup $\Gamma$, let $[P']$ be the isogeny class of $P'=[g_\infty, 1]\in X_\Gamma(K^{ab})$ corresponding to $(A,\imath,\bar\varphi)$ for some $\Gamma$-equivalence class $\bar\varphi$. Then Shimura's Reciprocity Law asserts that $\psi_{g_\infty}([ g_f])^\sigma=\psi_{g_\infty}([g_f{\rm Art}^{-1}(\sigma)])$, for all $\sigma\in\Gal(K^{ab}/K)$. We apply Theorem \ref{TeoGalact} and we obtain that $\rho_{(A,\imath,\varphi)}(\sigma)={\rm Art}^{-1}(\sigma\mid_{K^{ab}})$ modulo $\Gamma\cap \A_{K,f}^\times$, for any open subset $\Gamma$. Thus part $(iii)$ follows.

Finally, note first that the class $jK^\times$ is an element in $N_A/K^\times$ of order 2 and $N_A/K^\times=A_{K,f}^\times/K^\times\rtimes\langle jK^\times\rangle$. Let $\sigma_c\in\Aut(\C)$ denote complex conjugation automorphism. Recall that $A=(B\otimes\R)_{g_\infty}/\cO$, where the complex structure on $(B\otimes\R)_{g_\infty}$ is given by $h_{g_\infty}:\C\hookrightarrow\M_2(\R)$. It is easy to see that $h_{g_\infty}$ is the $\R$-extension of scalars of $\psi_{(A,\imath)}:K\hookrightarrow B$. Complex conjugation must be the unique automorphism $\gamma$ on $(B\otimes\R)$ such that $\gamma(h_{g_\infty}(z))=h_{g_\infty}(z^{\sigma_c})$, for all $z\in\C^\times$. Therefore $\gamma$ corresponds to conjugation by $j$ since $j^{-1}\psi_{(A,\imath)}(k)j=\psi_{(A,\imath)}(k^{\sigma_c})$, for all $k\in K^\times$. This implies that $A^{\sigma_c}=(B\otimes\R)_{g_\infty j}/\cO=(B\otimes\R)_{g_\infty }/\cO j$ and we have the following diagram (choosing $j^{-1}\in\cO$):
\[
\xymatrix{
A_{tor}\ar[d]^{\varphi}\ar[rr]_{P\mapsto P^{\sigma_c}}&&A_{tor}^{\sigma_c}\ar[rr]_{\simeq}\ar[d]&&A_{tor}^{\sigma_c}\ar[d]\ar[rr]_{\mu_{\sigma_c}}&& A_{tor}\ar[d]^{\varphi}\\
B/\cO\ar[rr]_{b\mapsto j^{-1} b j}&&B/j^{-1}\cO j \ar[rr]_{b\mapsto jb}&&B/\cO j \ar[rr]_{b\mapsto b}&&B/\cO
}
\]
Thus $\varphi(\mu_{\sigma_c}(P^{\sigma_c}))=\varphi(P)j$, which implies $\rho_{(A,\imath)}(\sigma_c)=jK^\times$. Since $\rho^N_{(A,\imath,\varphi)}$ maps exhaustively $\Gal(\bar \Q/K)$ to $A_{K,f}^\times/K^\times$ by $(iii)$ and $\Gal(\bar\Q/\Q)/\Gal(\bar \Q/K)\simeq\Gal(K/\Q)$ is generated by the image of $\sigma_c$, part $(iv)$ follows.
\end{proof}

\section{Abelian $\Q$-surfaces as factors of abelian varieties of $\GL_2$-type}\label{sec8}

In this section we will deal with abelian $\Q$-surfaces without complex multiplication. 
Thus, let $(A,\imath)$ be an abelian $\Q$-surface with QM by $\cO$ such that $\End^0(A,\imath)=\Q$. We fix a set of $\cO$-isogenies $\mu=\{\mu_\sigma:(A^\sigma,\imath^\sigma)\rightarrow(A,\imath),\;\sigma\in\Gal(\bar\Q/\Q)\}$. Given $\mu$, we can define the map:
\[
c_\mu:\Gal(\bar\Q/\Q)\times\Gal(\bar\Q/\Q)\longrightarrow\End^0(A,\imath)^\times=\Q^\times;\quad c_\mu(\sigma,\tau)=\frac{1}{\deg(\mu_{\sigma\tau})}\mu_\sigma\mu_\tau^\sigma\hat\mu_{\sigma\tau}.
\]
\begin{proposition}\label{propcocycle}
The map $c_\mu$ is a two cocycle in $Z^2(\Gal(\bar\Q/\Q),\Q^\times)$. Its class in $H^2(\Gal(\bar\Q/\Q),\Q^\times)$ does not depend on the choice of $\mu$ or the choice of $(A,\imath)$ in a fixed $\cO$-isogeny class. The class of $c_\mu$ in $H^2(\Gal(\bar\Q/\Q),\bar\Q^\times)$ coincides with the class of $B$ in $\Br(\Q)$, once we identify $\Br(\Q)\simeq H^2(\Gal(\bar\Q/\Q),\bar\Q^\times)$.
\end{proposition}
\begin{proof}
To prove that $c_\mu$ is a 2-cocycle one has to check that
\[
c_\mu(\sigma_2,\sigma_3)^{\sigma_1}c_\mu(\sigma_1,\sigma_2\sigma_3)=c_\mu(\sigma_1,\sigma_2)c_\mu(\sigma_1\sigma_2,\sigma_3), \quad\sigma_1,\sigma_2,\sigma_3\in\Gal(\bar\Q/\Q).
\]
We leave the details to the reader. If we choose $(A',\imath')$ in the $\cO$-isogeny class of $(A,\imath)$ with set of $\cO$-isogenies $\mu'=\{\mu_\sigma':((A')^\sigma,(\imath')^\sigma)\rightarrow(A',\imath'),\;\sigma\in\Gal(\bar\Q/\Q)\}$, we obtain that $c_\mu=c_{\mu'}\partial(\lambda)$, where 
\[
\lambda:\Gal(\bar\Q/\Q)\longrightarrow\Q^\times=\End^0(A,\imath)^\times;\quad
\lambda(\sigma)=\frac{1}{\deg(\mu_{\sigma}')\deg(\phi)}\mu_\sigma\circ\hat\phi^\sigma\circ\hat\mu'_{\sigma}\circ\phi,
\] 
for any $\cO$-isogeny $\phi:(A,\imath)\rightarrow(A',\imath')$.
Finally, the last assertion follows from \cite{Pyle} Theorem 2.1.
\end{proof}

Let us denote by $Q^\times$ the group $\bar\Q^\times$ with trivial Galois action. By a theorem due to Tate, \cite[Theorem 6.3]{Rib}, the cohomology group $H^2(\Gal(\bar\Q/\Q),Q^\times)$ is trivial. This implies that $c_\mu=\partial(\alpha)$, for some $\alpha:\Gal(\bar\Q/\Q)\rightarrow Q^\times$. Let $E_\alpha$ be the number field generated by the image of $\alpha$.

By means of the set of $\cO$-isogenies $\mu=\{\mu_\sigma;\;\sigma\in\Gal(\bar\Q/\Q)\}$ and a 1-cocycle $\alpha:\Gal(\bar\Q/\Q)\rightarrow Q^\times$ such that $c_\mu=\partial(\alpha)$,
we shall construct an abelian variety of $\GL_2$-type $A_{\GL_2}$ having $A$ as a factor or being contained in $A$. 

\begin{remark}\label{Remunoenotro}
In the particular case that $B=\M_2(\Q)$ and $A\simeq E\times E'$, where $E$ is an elliptic curve defined over $\Q$, we may obtain that $A_{\GL_2}=E$. Otherwise, we will obtain that $A$ is a factor of $A_{\GL_2}$.
\end{remark}

Choose a model over a Galois extension $M/\Q$ such that the $\cO$-isogenies $\mu_\sigma:(A^\sigma,\imath^\sigma)\rightarrow (A,\imath)$, with $\sigma\in\Gal(M/\Q)$, are also defined over $M$.
We consider the Weil restriction of scalars $X=\Res_\Q A$. It is an abelian variety defined over $\Q$ isomorphic over $\bar\Q$ to $\prod_{\sigma\in\Gal(M/\Q)} A^\sigma$. For any $\tau\in\Gal(M/\Q)$ and $o\in\cO$, the isogeny $\mu_\tau$ gives rise to an endomorphism $o\lambda_\tau\in\End(X)$ defined by:
\[
o\lambda_\tau:\prod_\sigma A^\sigma\longrightarrow\prod_\sigma A^\sigma,\quad(P_\sigma)_\sigma\longmapsto ((\imath(o)\mu_\tau)^{\sigma}(P_{\sigma\tau}))_{\sigma}.
\]
Since the Galois action on the points of $X$ is given by 
\[\begin{array}{ccc}
\prod_{\sigma\in\Gal(M/\Q)} A^\sigma&\stackrel{\gamma\in\Gal(\bar\Q/\Q)}{\longrightarrow}&\prod_{\sigma\in\Gal(M/\Q)} A^\sigma\\
(P_\sigma)_\sigma&\longmapsto&(P_\sigma^{\gamma})_{\pi(\gamma)\sigma},
\end{array}\]
where $\pi:\Gal(\bar\Q/\Q)\rightarrow\Gal(M/\Q)$ is the natural projection. We compute that, for $\tau\in\Gal(M/\Q)$, $\gamma\in\Gal(\bar\Q/\Q)$ and $o\in\cO$, 
\begin{eqnarray*}
o\lambda_\tau((P_\sigma)_\sigma^\gamma)&=&o\lambda_\tau((P_\sigma^{\gamma})_{\pi(\gamma)\sigma})=((\imath(o)\mu_{\tau})^{\pi(\gamma)\sigma}(P_{\sigma\tau}^\gamma))_{\pi(\gamma)\sigma}\\
&=&((\imath(o)\mu_\tau)^{\sigma}(P_{\sigma\tau}))_{\sigma}^{\gamma}=o\lambda_\tau((P_\sigma)_\sigma)^\gamma.
\end{eqnarray*}
Hence $o\lambda_\tau\in\End_\Q(X)$.

\begin{lemma}
We have that $\End^0_\Q(X)=\prod_{\sigma\in\Gal(M/\Q)}B\lambda_\tau$.
\end{lemma}
\begin{proof}
The abelian variety $X$ represents Weil's restriction of scalars functor $\Res_\Q(A)$, defined as the functor that maps a $\Q$-scheme $S$ to $A(S_M)$, where $S_M=S\otimes_\Q M$ is the $M$-scheme obtained from $S$ through extension of scalars. In the particular situation $S=X$, we obtain, 
\[
X(X)\otimes\Q=\End^0_\Q(X)=\Hom^0_M(X_M,A)=\prod_{\sigma\in\Gal(M/\Q)}\Hom^0_M(A^\sigma,A).
\]
Hence the result follows from the fact that $\Hom^0_M(A^\sigma,A)=\imath(B)\mu_\sigma$, which clearly maps to $B\lambda_\sigma$ under the above isomorphism. 
\end{proof}

\begin{remark}\label{remcocylambda}
Note that $(o\lambda_\tau)(o'\lambda_{\tau'})=c_\mu(\tau,\tau')(oo'\lambda_{\tau\tau'})$ for any $\tau,\tau'\in\Gal(M/\Q)$ and $o,o'\in\cO$. Indeed,
\begin{eqnarray*}
(o\lambda_\tau) (o'\lambda_{\tau'})((P_\sigma)_\sigma)&=&o\lambda_\tau(((\imath(o')\mu_{\tau'})^{\sigma}(P_{\sigma\tau'}))_{\sigma})\\
&=&((\imath(o)\mu_\tau)^{\sigma}((\imath(o')\mu_{\tau'})^{\sigma\tau}(P_{\sigma\tau\tau'})))_{\sigma}\\
&=&((\imath(o)\circ\mu_\tau\circ\imath^{\tau}(o')\circ\mu_{\tau'}^{\tau})^{\sigma}(P_{\sigma\tau\tau'}))_{\sigma}\\
&=&c_\mu(\tau,\tau')((\imath(oo')\mu_{\tau\tau'})^{\sigma}(P_{\sigma\tau\tau'}))_{\sigma}\\
&=&c_\mu(\tau,\tau')(oo'\lambda_{\tau\tau'})((P_\sigma)_\sigma).
\end{eqnarray*}
\end{remark}

Given the 2-cocycle $c_\mu$ associated to the set of $\cO$-isogenies $\mu=\{\mu_\sigma,\;\sigma\in\Gal(M/\Q)\}$ and the fixed $1$-cocycle $\alpha:\Gal(M/\Q)\rightarrow E_\alpha^\times$ such that $c_\mu=\partial(\alpha)$, we claim that there is a ring homomorphism 
\[
\psi:\End_\Q^0(X)\longrightarrow B\otimes_\Q E_\alpha,\quad b\lambda_\sigma\longmapsto b\otimes\alpha(\sigma).
\]
Indeed, for all $\sigma,\tau\in\Gal(M/\Q)$ and $b,b'\in B$,
\begin{eqnarray*}
\psi((b\lambda_\sigma)(b'\lambda_\tau))&=&\psi(c_\mu(\sigma,\tau)(bb'\lambda_{\sigma\tau}))=bb'\otimes c_\mu(\sigma,\tau)\alpha(\sigma\tau)\\
&=&(b\otimes\alpha(\sigma))(b'\otimes\alpha(\tau))=\psi(b\lambda_\sigma)\psi(b'\lambda_\tau),
\end{eqnarray*}
by Remark \ref{remcocylambda}.

Clearly the map $\psi$ is surjective by definition. Let us consider the abelian subvariety $X_0/\Q$ defined by $X\supseteq X_0:=\{P\in X:\lambda( P)=0, \;\mbox{for all}\;\lambda\in N\}$, where $N$ is the subgroup of endomorphisms $N=\ker(\psi)\cap\End_\Q(X)$. By construction $\End^0_\Q(X_0)=\End^0_\Q(X)/\ker(\psi)=B\otimes_\Q E_\alpha$. Since $\dim(X)=2[M:\Q]$ and $\dim_\Q(\End^0_\Q(X))=4[M:\Q]=2\dim(X)$, we deduce 
\[
4[E_\alpha:\Q]=\dim_\Q(\End^0_\Q(X_0))=2\dim(X_0).
\]
\begin{lemma}
We have that $B\otimes_\Q E_\alpha=\M_2(E_\alpha)$.
\end{lemma}
\begin{proof}
By Proposition \ref{propcocycle}, the 2-cocycle $c_\mu$ represents the class of $B$ in the Brauer group $\Br(\Q)$. Since $c_\mu$ is a coboundary when extended to $E_\alpha$, it represents the trivial element of the Brauer group $\Br(E_\alpha)$. The fact that the cocycle representing $B$ is trivial when extended to $E_\alpha$ implies the assertion (see \cite[\S 2]{Pyle}).
\end{proof}

The above lemma implies that $X_0$ is isogenous over $\Q$ to an abelian surface of the form $A_{\GL_2}^2$, where $\End_\Q^0(A_{\GL_2})=E_\alpha$. Clearly, $\dim(A_{\GL_2})=\dim(X_0)/2=[E_\alpha:\Q]$, thus $A_{\GL_2}$ is an abelian surface of $\GL_2$-type. More precisely, let $\pi$ be the an element of the set
\[
\pi\in\End_\Q(X_0)\cap\left\{\left(\begin{array}{cc}x&0\\0&0\end{array}\right)\in B\otimes_\Q E_\alpha=\End^0_\Q(X_0),\;x\in E_\alpha\right\}.
\]
Then $A_{\GL_2}:=\pi(X_0)\subset X_0$.

Since $A_{\GL_2}\subset X_0\subseteq X=\Res_\Q(A)$, we have that $A$ is a factor of $A_{\GL_2}$ (or $A_{\GL_2}$ is a factor of $A$, see Remark \ref{Remunoenotro}). In particular, $\End^0_{\bar\Q}(A_{\GL_2})=\M_{d/2}(B)$, where $d=[E_\alpha:\Q]$ (we write formally $\M_{1/2}(\M_2(\Q))=\Q$ in case $E_\alpha=\Q$ and $B=\M_2(\Q)$).

\subsection{Galois representations attached to abelian varieties of $\GL_2$-type}

Let $A_{\GL_2}/\Q$ be an abelian variety of dimension $d$ of $\GL_2$-type such that $\End^0_{\bar\Q}(A_{\GL_2})=\M_{d/2}(B)$. By definition $\End^0_\Q(A_{\GL_2})=E$, a field extension of $\Q$ of degree $d$. 

First we recall the classical construction of the Galois representation attached to $A_{\GL_2}$:
For any prime $\ell$, the Tate module $T_\ell(A_{\GL_2})=\Hom((A_{\GL_2})_{tor},\Q_\ell/\Z_\ell)$ is a free $\Z_\ell$-module of rank $2d$. Its extension of scalars $V_\ell(A_{\GL_2}):=T_\ell(A_{\GL_2})\otimes\Q$ has a natural action of $\End^0_\Q(A_{\GL_2})=E$, which provides it with a structure of rank 2 $E\otimes\Q_\ell$-module. Given a basis $\varphi_1,\varphi_2\in V_\ell(A_{\GL_2})$ as a $E\otimes\Q_\ell$-module, we obtain a representation
\[
\rho^\ell_{\GL_2}:\Gal(\bar\Q/\Q)\longrightarrow\GL_2(E\otimes\Q_\ell)^{opp},
\]
defined by $(\varphi_1(P^\sigma),\varphi_2(P^\sigma))=(\varphi_1( P),\varphi_2( P))\rho^\ell_{\GL_2}(\sigma)$, for all $P\in(A_{\GL_2})_{tor}$ and $\sigma\in\Gal(\bar\Q/\Q)$.

Note that we can consider the product of subgroups $G(\A_f)E_\alpha^\times$ inside $\prod_\ell\GL_2(E_\alpha\otimes\Q_\ell)^{opp}$, where $E_\alpha$ is embedded in $\prod_\ell\GL_2(E_\alpha\otimes\Q_\ell)^{opp}$ diagonally and $G(\A_f)$ through the monomorphism 
\[
G(\A_f)=(\hat\cO^{opp}\otimes\Q)^\times\hookrightarrow \prod_\ell(B_\ell^{opp}\otimes E_\alpha)^\times\simeq \prod_\ell\GL_2(E_\alpha\otimes\Q_\ell)^{opp}.
\]
The following result relates this representation $\rho^\ell_{\GL_2}$ with the representation $\rho^\mu_{(A,\imath,\varphi)}$ introduced in \S \ref{sec2}:
\begin{theorem}\label{eqrep}
Let $(A,\imath)$ be an abelian $\Q$-surface with QM by $\cO$. Assume that the abelian variety of $GL_2$-type $A_{GL_2}/\Q$ has been constructed by means of $(A,\imath)$, $\mu=\{\mu_\sigma:(A^\sigma,\imath^\sigma)\rightarrow(A,\imath),\;\sigma\in\Gal(\bar\Q/\Q)\}$ and $\alpha:\Gal(\bar\Q/\Q)\rightarrow E_\alpha^\times$. Then, there exist $(E_\alpha\otimes\Q_\ell)$-basis of the Tate modules of $A_{\GL_2}$ at every prime $\ell$ such that the product of the $\ell$-adic representations $\prod_\ell\rho^\ell_{\GL_2}=:\hat\rho_{\GL_2}$ factors through
\[\xymatrix{
\hat\rho_{\GL_2}:\Gal(\bar\Q/\Q)\ar[rr]^{\rho^\mu_{(A,\imath,\varphi)}\alpha^{-1}\qquad}&&G(\A_f)E_\alpha^\times\subset\prod_\ell\GL_2(E_\alpha\otimes\Q_\ell)^{opp}.
}\]
Namely, $\hat\rho_{\GL_2}(\sigma)=\rho^\mu_{(A,\imath,\varphi)}(\sigma)\alpha^{-1}(\sigma)\in G(\A_f)E_\alpha^\times\subseteq\prod_\ell\GL_2(E_\alpha\otimes\Q_\ell)^{opp}$, for all $\sigma\in \Gal(\bar\Q/\Q)$.
\end{theorem}
\begin{proof}
Fix a model $(A,\imath)$ over some Galois extension $M$ as above.
In order to construct $A_{\GL_2}$, we consider the restriction of scalars $X=\Res_\Q(A)$. The $\cO$-module isomorphism $\varphi$ provides a $\cO$-module epimorphism
\[
\bar\varphi:X_{tor}=\prod_{\sigma\in\Gal(M/\Q)}A_{tor}^\sigma\stackrel{pr}{\longrightarrow} A_{tor}\stackrel{\varphi}{\longrightarrow} B/\cO,
\]
where $pr$ is the natural projection.
Let $\pi:\Gal(\bar\Q/\Q)\rightarrow\Gal(M/\Q)$ be the natural quotient morphism, we compute, for all $\gamma\in\Gal(\bar\Q/\Q)$,
\begin{eqnarray*}
\bar\varphi(\lambda_{\pi(\gamma)}(P_\sigma)_\sigma^\gamma)&=&\bar\varphi(\mu_{\pi(\gamma)}^{\pi(\gamma)\sigma}(P_{\sigma\pi(\gamma)}^\gamma)_{\pi(\gamma)\sigma})=\varphi(\mu_{\pi(\gamma)}(P^{\gamma}))\\
&=&\varphi( P)\rho^\mu_{(A,\imath,\varphi)}(\gamma)=\bar\varphi( P_\sigma)_\sigma\rho^\mu_{(A,\imath,\varphi)}(\gamma).
\end{eqnarray*}
Given $\alpha$ and the corresponding subvariety $X_0\subseteq X$, since $A\subseteq X_0$, the restriction of $\bar\varphi$ to $X_0$ gives rise to an epimorphism $\bar\varphi_0:(X_0)_{tor}\rightarrow B/\cO$. Moreover, it satisfies
\begin{equation}\label{eqnthmGal}
\bar\varphi_0(\alpha(\gamma)Q^\gamma)=\bar\varphi_0(Q)\rho^\mu_{(A,\imath,\varphi)}(\gamma),\quad\mbox{for all}\;Q\in(X_0)_{tor}.
\end{equation}
Note that, once that we have chosen a $\hat\Z$-basis for $\hat\cO$, such a map $\bar\varphi_0$ provides 4 linearly independent elements of the Tate module $\hat T(X_0)=\Hom((X_0)_{tor},\Q/\Z)$. More precisely, if $\{e_1,e_2,e_3,e_4\}$ is a $\hat\Z$-basis of $\hat\cO$, there exists $\varphi_i\in\hat T(X_0)$ such that
\[
\bar\varphi_0(Q)=\varphi_1(Q)e_1+\varphi_2(Q)e_2+\varphi_3(Q)e_3+\varphi_4(Q)e_4,
\]
for all $Q\in (X_0)_{tor}$. By Remark \ref{remcocylambda}, the endomorphisms in $E_\alpha$ commute with the endomorphisms in $\imath(\cO)$. Since $\dim_Q(B\otimes_\Q E_\alpha)=2\dim(X_0)$, the elements $\varphi_i\in\hat T(X_0)$ provide an isomorphism
\[
\varphi_\alpha:\hat\cO\otimes E_\alpha\stackrel{\simeq}{\longrightarrow}\hat V(X_0):=\hat T(X_0)\otimes\Q;\qquad (\sum_ix_ie_i)\otimes z\longmapsto(Q\mapsto\sum_i x_i\varphi_i(z Q)),
\]
with $x_i\in\hat\Z$, $z\in E_\alpha$ and $Q\in(X_0)_{tor}$. By \eqref{eqnthmGal}, we have that, for all $Q\in(X_0)_{tor}$, $o\otimes z\in\hat\cO\otimes E_\alpha$ and $\gamma\in\Gal(\bar\Q/\Q)$,
\[
\varphi_\alpha(o\otimes z)(Q^\gamma)=\varphi_\alpha(o\rho^\mu_{(A,\imath,\varphi)}(\gamma)\otimes \alpha^{-1}(\gamma)z)(Q).%=\varphi_\alpha((o\otimes z)\rho^{\mu}_{(A,\imath,\varphi)}(\gamma)\alpha^{-1}(\sigma))(Q).
\]
Thus the action of $\gamma\in\Gal(\bar\Q/\Q)$ on $\hat V(X_0)\simeq\hat\cO\otimes E_\alpha$ is given by right multiplication by $\rho^{\mu}_{(A,\imath,\varphi)}(\gamma)\alpha^{-1}(\sigma)$.

Finally, we know that $A_{\GL_2}=\pi(X_0)\subset X_0$, where $\pi\in\cO\otimes E_\alpha$. Hence $\varphi_\alpha$ provides an isomorphism $\hat V(A_{\GL_2})_{tor}\simeq\pi(\hat\cO\otimes E_\alpha)$. Since right multiplication by $G(\A_f)E_\alpha$ commutes with left multiplication by $\cO\otimes E_\alpha$, we conclude that the action of $\gamma\in\Gal(\bar\Q/\Q)$ on $\hat V(A_{\GL_2})\simeq\pi(\hat\cO\otimes E_\alpha)$ is also given by right multiplication by $\rho^{(\mu,\alpha)}_{(A,\imath,\varphi)}(\gamma)$. The fixed the isomorphism $\cO_\ell\otimes E_\alpha\simeq \M_2(E_\alpha\otimes\Q_\ell)$ provides a basis of the Tate module $\hat V(A_{\GL_2})\simeq\pi(\cO_\ell\otimes E_\alpha)$ as a rank 2 $(E_\alpha\otimes\Q_\ell)$-module. It is clear that the Galois action on $V_\ell(A_{\GL_2})$ is given by right multiplication by the $\ell$-adic component of $\rho^{\mu}_{(A,\imath,\varphi)}(\gamma)\alpha^{-1}(\sigma)$, once we have identified $G(\Q_\ell)E_\alpha$ inside $\GL_2(\Q_\ell\otimes E_\alpha)^{opp}$. Thus the result follows.
\end{proof}

The following result, which is a direct consequence of the above theorem, is well known by the experts.
\begin{corollary}
Let $\hat\rho_{\GL_2}=\prod_\ell\rho^\ell_{\GL_2}$ be the Galois representation attached to the abelian variety of $\GL_2$-type constructed by means of $(A,\imath)$, $\mu$ and $\alpha$. Then its determinant is given by
\[
\det(\hat\rho_{\GL_2})(\sigma)=\chi(\sigma)\frac{\deg(\mu_\sigma)}{\alpha^{2}(\sigma)},
\]
where $\chi$ is the cyclotomic character.
\end{corollary}
\begin{proof}
Since the determinant does not depends on the choice of the basis for the Tate module, we choose the basis of Theorem \ref{eqrep}. Then the result follows from Theorem \ref{eqrep} and Theorem \ref{normrho}.
\end{proof}

\begin{remark}
After the proof of Serre's conjecture on representations of $\Gal(\bar\Q/\Q)$ \cite[3.2.4?]{Ser}, $\rho^\ell_{\GL_2}$ is the $\ell$-adic Galois representation associated with a modular newform in $S_1(N,\epsilon)$, with $N\in\N$ and $\epsilon(\sigma)=\frac{\deg(\mu_\sigma)}{\alpha^{2}(\sigma)}$, for all $\sigma\in\Gal(\bar\Q/\Q)$.
\end{remark}

Recall that the representation $\rho_{(A,\imath,\varphi)}:\Gal(\bar\Q/\Q)\rightarrow G(\A_f)/\Q^\times$ introduced in \S \ref{sec2} is the projectivization of $\rho_{(A,\imath,\varphi)}^\mu$ modulo $\Q^\times$. Applying Theorem \ref{eqrep}, we deduce the following corollary:
\begin{corollary}\label{eqrepcor}
The representation $\rho_{(A,\imath,\varphi)}$ is the projectivization of the classical representation $\hat\rho_{\GL_2}:\Gal(\bar\Q/\Q)\rightarrow G(\A_f)E_\alpha^\times$ attached to the abelian variety of $\GL_2$-type $A_{\GL_2}$ modulo $\End_\Q^0(A_{\GL_2})^\times=E_\alpha^\times$.
\end{corollary}

\bibliographystyle{plain}
\bibliography{GalActIsoCla}

\end{document}